\documentclass[11pt,reqno]{amsart}
\usepackage{color}
\usepackage[english]{babel}
\usepackage{amsfonts}
\usepackage{amsmath}
\usepackage{amsthm}
\usepackage{amssymb}
\usepackage[misc]{ifsym}
\usepackage{bbm}
\usepackage{mathrsfs}
\usepackage{esint}
\usepackage{faktor}
\usepackage[utf8]{inputenc}
\usepackage{afterpage}
\usepackage[left=2.9cm,right=2.9cm,top=2.8cm,bottom=2.8cm]{geometry}
\usepackage{graphicx}
\usepackage[pagewise,displaymath,mathlines]{lineno}
\usepackage{enumitem}
\usepackage[colorlinks=true,urlcolor=blue,citecolor=blue,linkcolor=blue,linktocpage,pdfpagelabels,bookmarksnumbered,bookmarksopen]{hyperref}

\newcommand{\N}{\mathbb{N}}

\newcommand{\R}{\mathbb{R}}

\newcommand{\Div}{\mathrm{div} \, }
\newcommand{\dx}{\, {\rm d} x}
\newcommand{\dy}{\, {\rm d} y}
\newcommand{\dt}{\, {\rm d} t}
\newcommand{\ds}{\, {\rm d} s}
\newcommand{\dr}{\, {\rm d} r}
\newcommand{\dtau}{\, {\rm d} \tau}
\newcommand{\eps}{\varepsilon}
\renewcommand{\phi}{\varphi}

\newtheorem{lemma}{Lemma}[section]
\newtheorem{thm}[lemma]{Theorem}
\newtheorem{prop}[lemma]{Proposition}

\theoremstyle{definition}
\newtheorem{defi}[lemma]{Definition}
\newtheorem{rmk}[lemma]{Remark}
\newtheorem{ex}[lemma]{Example}

\numberwithin{equation}{section}

\DeclareMathOperator*{\supp}{supp}

\DeclareMathOperator*{\dom}{dom}

\begin{document}

\title[Singular $\Phi$-Laplacian systems in $\R^N$]{Existence, uniqueness, and decay results \\ for singular $\Phi$-Laplacian systems in $\R^N$}
\author[L. Gambera]{Laura Gambera}
\address[L. Gambera]{Dipartimento di Matematica e Informatica, Universit\`a degli Studi di Catania, Viale A. Doria 6, 95125 Catania, Italy}
\email{laura.gambera@studium.unict.it}
\author[U. Guarnotta]{Umberto Guarnotta}
\address[U. Guarnotta]{Dipartimento di Matematica e Informatica, Universit\`a degli Studi di Catania, Viale A. Doria 6, 95125 Catania, Italy}
\email{umberto.guarnotta@studium.unict.it}

\maketitle

\begin{abstract}
Existence of solutions to a $\Phi$-Laplacian singular system is obtained via shifting method and variational methods. A priori estimates are furnished through De Giorgi's technique, Talenti's rearrangement argument, and exploiting the weak Harnack inequality, while decay of solutions is obtained via comparison with radial solutions to auxiliary problems. Finally, uniqueness is investigated, and a D\'iaz-Sa\'a type result is provided. 
\end{abstract}

{
\let\thefootnote\relax
\footnote{\Letter \quad Corresponding author: Umberto Guarnotta (umberto.guarnotta@studium.unict.it).}
}
\setcounter{footnote}{0}

\begin{center}
\begin{minipage}{11cm}
\begin{small}
\tableofcontents
\end{small}
\end{minipage}
\end{center}

\section{Introduction, main results, and examples}

In this paper we study the problem

\begin{equation}
\label{prob}
\tag{P}
\left\{
\begin{alignedat}{2}
-\Delta_{\Phi} u &= h(x)f(u,v) &&\quad \mbox{in}\;\; \R^N, \\
-\Delta_{\Psi} v &= k(x)g(u,v) &&\quad \mbox{in}\;\; \R^N, \\
u,v &>0 &&\quad \mbox{in}\;\; \R^N, \\
u(x),v(x) &\to 0 &&\quad \mbox{as}\;\; |x| \to +\infty.
\end{alignedat}
\right.
\end{equation}
The principal parts of the operators appearing in \eqref{prob} are defined as
\begin{equation*}
\Delta_\Phi u := \Div\left(\phi(|\nabla u|)\frac{\nabla u}{|\nabla u|}\right), \quad \Delta_\Psi v := \Div\left(\psi(|\nabla v|)\frac{\nabla v}{|\nabla v|}\right),
\end{equation*}
where $\phi=\Phi'$ and $\psi=\Psi'$ (see \hyperlink{H1}{${\rm (H_1)}$} below). Moreover, $f,g:\R_+\times\R_+\to\R_+$ are continuous functions. We suppose the following hypotheses (please refer to Section 2 for notation).

\begin{enumerate}[label={${\rm (H_1)}$}]
\hypertarget{H1}{}
\item \label{ellipticity}
The Young functions $\Phi,\Psi$ are of class $C^2$ and
\begin{equation*}
\begin{split}
&0 < i_\phi := \inf_{t>0}\frac{t\phi'(t)}{\phi(t)} \leq \sup_{t>0}\frac{t\phi'(t)}{\phi(t)} =: s_\phi < +\infty, \\
&0 < i_\psi := \inf_{t>0}\frac{t\psi'(t)}{\psi(t)} \leq \sup_{t>0}\frac{t\psi'(t)}{\psi(t)} =: s_\psi < +\infty.
\end{split}
\end{equation*}
Moreover, $\max\{s_\Phi,s_\Psi\}<N$.
\end{enumerate}
\begin{enumerate}[label={${\rm (H_2)(\roman*)}$}]
\hypertarget{H2}{}
\item \label{varstruct} There exists $H:\R^N\times\R_+\times\R_+\to\R$ such that
\begin{itemize}
\item $H(\cdot,s,t)$ is measurable for all $s,t\in\R_+$;
\item $H(x,\cdot,\cdot)$ is of class $C^1$ for all $x\in\R^N$;
\item $\partial_s H(x,s,t) = h(x)f(s,t)$ for all $(x,s,t)\in\R^N\times\R_+\times\R_+$;
\item $\partial_t H(x,s,t) = k(x)g(s,t)$ for all $(x,s,t)\in\R^N\times\R_+\times\R_+$;
\item $H(\cdot,0,0)\equiv 0$.
\end{itemize}
\item \label{growthcond} For suitable Young functions $\Upsilon_i,\Gamma_i$, $i=1,2$, with $\Upsilon_1\ll\Phi$ and $\Gamma_2\ll\Psi$, one has
\begin{equation*}
\begin{split}
f(s,t) &\leq C \left[ (s^{-\alpha}+1)\left(\frac{\Upsilon_2(t)}{t}+1\right) + \frac{\Upsilon_1(s)}{s} \right], \\
g(s,t) &\leq C \left[ (t^{-\beta}+1)\left(\frac{\Gamma_1(s)}{s}+1\right) + \frac{\Gamma_2(t)}{t} \right], \\
\end{split}
\end{equation*}
for all $(s,t)\in\R_+\times\R_+$, where $\alpha,\beta\in(0,1)$. Moreover,
\begin{equation}
\label{prodcond}
\overline{\Phi} \ll \Psi\circ\Upsilon_2^{-1}\circ\overline{\Upsilon}_2 \quad \mbox{and} \quad \overline{\Psi} \ll \Phi\circ\Gamma_1^{-1}\circ\overline{\Gamma}_1.
\end{equation}
\item \label{degiorgicond} There exist $\delta_1>\frac{N}{i_\Phi}$, $\delta_2>\frac{N}{i_\Psi}$ such that
\begin{equation}
\label{powercond1}
t^{(s_{\Upsilon_1}-i_\Phi)\delta_1} < \Phi_*, \quad t^{(s_{\Gamma_2}-i_\Psi)\delta_2} < \Psi_*,
\end{equation}
\begin{equation}
\label{powercond2}
t^{\delta_1}\circ\overline{\Upsilon}_2^{-1}\circ\Upsilon_2 < \Psi_*, \quad t^{\delta_2}\circ\overline{\Gamma}_1^{-1}\circ\Gamma_1 < \Phi_*.
\end{equation}
\item \label{harnackcond} For all $M>0$ there exists $c_M>0$ such that
\begin{equation*}
f(s,t) \geq c_M t^{\nu_2} \quad \mbox{and} \quad g(s,t) \geq c_M s^{\nu_1}
\end{equation*}
for all $s,t\in(0,M]$ and opportune $\nu_1,\nu_2>0$ satisfying
\begin{equation}
\label{exponents}
\nu_1\nu_2<\min\left\{\frac{i_\Phi}{s_{\overline{\Phi}}},\frac{N(i_\Phi-1)}{N-i_\Phi}\right\}\min\left\{\frac{i_\Psi}{s_{\overline{\Psi}}},\frac{N(i_\Psi-1)}{N-i_\Psi}\right\}.
\end{equation}
\item \label{weightscond} The weights $h,k:\R^N\to\R_+$ belong to $L^1(\R^N) \cap L^\infty(\R^N)$ and satisfy $h(x),k(x)\to 0$ as $|x|\to+\infty$; in addition, for every $r>0$ there exists $c_r>0$ such that $\min\{h,k\} \geq c_r$ in $B_r$.
\end{enumerate}
\begin{enumerate}[label={${\rm (H_3)(\roman*)}$}]
\hypertarget{H3}{}
\item \label{decaycond} There exist $\theta_1>N+\alpha\frac{N-i_\Phi}{i_\Phi-1}$ and $\theta_2>N+\beta\frac{N-i_\Psi}{i_\Psi-1}$ such that
\begin{equation*}
h(x) \leq C|x|^{-\theta_1} \quad \mbox{and} \quad k(x) \leq C|x|^{-\theta_2} \quad \mbox{for all} \;\; x\in\R^N.
\end{equation*}
\item \label{diazsaacond} The functions $s\mapsto\frac{f(s,t)}{s^{i_\phi}}$ and $t\mapsto\frac{g(s,t)}{t^{i_\psi}}$ are strictly decreasing for all $t\in\R_+$ and all $s\in\R_+$, respectively.
\end{enumerate}

\begin{rmk}
Let us briefly comment the hypotheses.
\begin{itemize}
\item \ref{ellipticity} is the ellipticity of the principal parts of the operators in \eqref{prob};
\item \ref{varstruct} represents the variational structure of \eqref{prob};
\item \ref{growthcond} is a $(\Phi,\Psi)$-sub-linear growth condition (in particular, \eqref{prodcond} pertains the mixed terms), ensuring the coercivity of the energy functional (see Lemma \ref{functprops});
\item \ref{degiorgicond} is a sub-criticality condition on the pure terms (see \eqref{powercond1}) and on the mixed terms (see \eqref{powercond2}), exploited to infer $L^\infty$ estimates (see Lemma \ref{supest} and Remark \ref{ondegiorgicond});
\item \ref{harnackcond} is a condition on the lower bounds of $f,g$, which implies a local lower bound for solutions to \eqref{prob} (see Lemma \ref{belowest} and Remark \ref{onharnackcond});
\item \ref{weightscond} summarizes the main properties of the weights $h,k$;
\item \ref{decaycond} is a quantitative information on the decay of the weights, used in Lemma \ref{decay};
\item \ref{diazsaacond} is a D\'iaz-Sa\'a type condition employed in Theorem \ref{uniqueness} (see Remark \ref{ondiazsaacond}).
\end{itemize}
\end{rmk}

\begin{rmk}
\label{ondegiorgicond}
Hypothesis \ref{degiorgicond} is not optimal; it is used only in \eqref{hatfunifbound} and \eqref{improvement2}. The optimal conditions for these arguments, useful also to generalize Lemma \ref{talenti}, are \cite[formula (3.9)]{BCM} and \cite[formulas (2.20)--(2.21)]{C2}, which are given in terms of weak Orlicz spaces and rearrangements, respectively. Anyway, in the power case, i.e., $\Phi(t):=t^p$ and $\Psi(t):=t^q$, \eqref{powercond1} reduces to the $p$-sub-critical growth of $\Upsilon_1$, which is the natural condition to meet in order to ensure the validity of Lemma \ref{supest} in the case of a single equation; we can reduce to this case since \eqref{powercond2} allows us to treat the reaction terms in $v$ as a frozen $\hat{f}(x)$ (see Lemma \ref{supest}).
\end{rmk}

\begin{rmk}
\label{onharnackcond}
Hypothesis \ref{harnackcond} comes mainly from the weak Harnack inequality \cite[Theorem 1.4]{BHHK} and it is used only in Lemma \ref{belowest}; in the case that $i_\Phi'=s_{\overline{\Phi}}$ and $i_\Psi'=s_{\overline{\Psi}}$ hold true (for instance in the power case), \eqref{exponents} is equivalent to the (more restrictive, in general) condition
\begin{equation*}
\nu_1\nu_2<(i_\Phi-1)(i_\Psi-1),
\end{equation*}
exploited to obtain estimates from below in the $p$-Laplacian case (cf., for instance, \cite[Lemma 3.8]{GMM}).
\end{rmk}

\begin{rmk}
\label{ondiazsaacond}
Hypotheses \ref{decaycond}--\ref{diazsaacond} are used only in the uniqueness result: see Theorem \ref{uniqueness}. Here we emphasize the fact that the condition \ref{diazsaacond} is optimal, even for $(p,q)$-Laplacian problems in bounded domains (say $\Omega$), in the sense that the indices $i_\phi$ and $i_\psi$ cannot be increased. Indeed, consider the problem
\begin{equation*}
\left\{
\begin{alignedat}{2}
-\Delta_p u -\Delta_q u &= \lambda u_+^{r-1} \quad &&\mbox{in} \;\; \Omega, \\
u&=0 \quad &&\mbox{on}\;\; \partial \Omega,
\end{alignedat}
\right.
\end{equation*}
being $1<p<r<q<+\infty$ and $\lambda>0$ large enough: the energy functional associated with this problem is coercive, has a local minimizer in $u=0$, and a global minimizer $u^*\neq0$, so there exists a critical point $\overline{u}$ given by the mountain pass theorem (see, e.g., \cite[Theorem 5.40]{MMP}), implying that the problem has two nontrivial solutions $u^*,\overline{u}$.
\end{rmk}

We highlight that problem \eqref{prob} possesses several combined features, including:
\begin{itemize}
\item the principal parts of the operators have unbalanced growth;
\item the reaction terms possess natural growth with respect to the principal parts;
\item the problem is set in the whole space $\R^N$;
\item a pointwise decay is required.
\end{itemize}
These features compel us to face several issues, exploiting in particular a suitable functional framework and some fine a priori estimates. More precisely, the unbalanced growth of both the operators and the reaction, as well as the setting $\R^N$, requires the usage of the Beppo Levi-Orlicz spaces (see Section 2 below); on the other hand, the natural decay of functions in such spaces is only measure-theoretical, in the sense that
\begin{equation*}
u\in \mathcal{D}^{1,\Phi}_0(\R^N) \quad \Rightarrow \quad |\{x\in\R^N: \, |u(x)| \geq \eps\}| < +\infty \quad \forall \eps>0.
\end{equation*}
Accordingly, we have to bound the solutions, from both above and below, and compare them with appropriate decaying radial functions, in order to prove a pointwise decay of the solutions.

Firstly, we obtain decay estimates for problems patterned after \eqref{prob}, under hypothesis \ref{ellipticity}. Following an argument developed for the $p$-Laplacian by \'Avila and Brock \cite{AB}, we construct radial solutions $w=w(|x|)$ to problems driven by the $\Phi$-Laplacian operator, and prove the asymptotic estimate
\begin{equation*}
w(x) \sim \int_{|x|}^{+\infty} \phi^{-1}(s^{1-N}) \ds \quad \mbox{as} \quad |x|\to+\infty,
\end{equation*}
which is consistent with the case of the $p$-Laplacian operator (see Remark \ref{decayest}).

Secondly, we show the existence of a weak solution to \eqref{prob}, provided \hyperlink{H1}{${\rm (H_1)}$}--\hyperlink{H2}{${\rm (H_2)}$} are satisfied. The proof consists of three steps: (i) regularization of the problem; (ii) a priori estimates; (iii) any distributional solution is a weak one. Step (i) is performed by the shifting technique, reducing the analysis to regular (i.e., non-singular) problems \eqref{regularprob}, whose existence of solutions is ensured by the direct methods of Calculus of Variations (Theorem \ref{regexistence}). Step (ii) provides uniform a priori estimates, in order to gain compactness and pass to the limit $\eps\to 0^+$: after proving energy estimates (see Lemma \ref{energyest}), we show that the solutions are essentially bounded via modern, general versions of Talenti's rearrangement argument \cite{C2} and De Giorgi's technique \cite{BCM} (vide Lemmas \ref{talenti}--\ref{supest}), and then we produce local bounds from below by exploiting a recent weak Harnack inequality in Orlicz spaces \cite{BHHK} (see Lemma \ref{belowest}). The estimates in step (ii) allow to get a distributional solution to \eqref{prob} (Lemma \ref{distrsol}), which is actually a weak one by step (iii), performed by using a regularization-localization argument (Theorem \ref{weaksol}).

Finally, supposing also \hyperlink{H3}{${\rm (H_3)}$}, we prove uniqueness of weak solutions by adapting an argument by D\'iaz and Sa\'a \cite{DS} (cf.~also \cite{CDS}); see Theorem \ref{uniqueness}. We point out that, unlike the aforementioned papers, here the operator is not homogeneous and we treat systems instead of single equations. The proof exploits the previous information regarding both the decay estimate and the sup-bounds.

Singular equations in exterior domains are a good starting point to investigate problems in the whole space: the former have been studied in \cite{KS}, which investigates asymptotically linear nonlinearities in the semi-linear case, \cite{CDS}, analyzing model singular nonlinearities in the $p$-Laplacian case, and \cite{CF}, that generalizes the previous result to strongly singular nonlinearities. Pertaining singular equations in the whole space, we limit ourselves to cite here only \cite{GS}; other results can be found in the survey \cite{GLM1}.

One of the first contributions in the topic of singular quasi-linear systems is represented by \cite{MMM}, which studies singularities of type $s^{-\alpha}$ (resp., $t^{-\gamma}$ for the second equation) instead of the model case $s^{-\alpha}t^{\beta}$ (resp., $s^{\gamma}t^{-\delta}$). A systematic treatment of particular parametric singular systems was furnished in \cite{SARZ}. A more recent result, regarding singular convective systems on $\R^N$, has been obtained in \cite{GMM}. For further contributions, we address the reader to the survey \cite{GLM2}, in which the present work was announced (see \cite[p.12]{GLM2}).

Another relevant aspect of the present work is represented by the Beppo Levi-Orlicz setting, which generalizes some previous results in different directions. Existence of solutions for a singular (and convective, that is, the reaction depends on the gradient of the solution) problem set in $\R^N$ and driven by an operator patterned after the $(p,q)$-Laplacian can be found in \cite{GG}. Existence results of solutions to Dirichlet problems in bounded domains driven by the more general $\Phi$-Laplacian (thus involving Sobolev-Orlicz spaces) can be found, for instance, by \cite{CGL} (see also \cite{CGP}) and \cite{CGSS} for equations, as well as \cite{GCS} for systems (in which the authors prove also uniqueness of solutions via a D\'iaz-Sa\'a type argument); here we have reached similar results in the whole $\R^N$. Nonetheless, we have proved the decay of the fundamental solution of the $\Phi$-Laplacian (as well as the solution to $-\Delta_\Phi u = \mu|x|^{-l}$ in $\R^N\setminus\{0\}$, being $\mu>0$ and $l>N$; see Lemmas \ref{radialsub}--\ref{radialsuper} below), extending the results for the $p$-Laplacian operator given by \cite{AB}. For an introduction to Orlicz spaces we refer to \cite{KJF,M,RR,HH}; we extensively used some optimal results about the Sobolev conjugate of a Young function and the related embedding \eqref{embedding}, a Talenti type rearrangement argument, and the boundess of solutions, obtained by Cianchi et al. \cite{C,C2,BCM}.

\begin{ex}[The power case]
\label{powercase}
Consider $\Phi(t):=t^p$, $\Psi(t):=t^q$,
\begin{equation*}
\begin{split}
f(u,v) := (1-\alpha)u^{-\alpha}v^{1-\beta} + r_1 u^{r_1-1} + \hat{r} u^{\hat{r}-1}v^{\hat{s}}, \\
g(u,v) := (1-\beta)u^{1-\alpha}v^{-\beta} + s_2 v^{s_2-1} + \hat{s} u^{\hat{r}}v^{\hat{s}-1}, \\
\end{split}
\end{equation*}
being $0<\alpha,\beta<1\leq \hat{r},\hat{s}$ and $r_1,s_2>1$; suppose also that there exist $r_2,s_1\geq 2$ such that
\begin{equation}
\label{youngcond}
\frac{\hat{r}-1}{r_1-1}+\frac{\hat{s}}{r_2-1}\leq 1, \quad \frac{\hat{r}}{s_1-1}+\frac{\hat{s}-1}{s_2-1}\leq 1.
\end{equation}
Moreover, take $h$ and $k$ to be equal to $w:=\frac{1}{1+|x|^\theta}$, $\theta>0$. Setting $\Upsilon_i(t):=t^{r_i}$, $\Gamma_i(t):=t^{s_i}$, $i=1,2$, and
\begin{equation*}
H(x,u,v) := w(x)\left[ u^{1-\alpha}v^{1-\beta} + u^{r_1} + v^{s_2} + u^{\hat{r}}v^{\hat{s}} \right],
\end{equation*}
we have:
\begin{itemize}
\item \ref{ellipticity} is satisfied if and only if $1<p,q<N$;
\item \ref{varstruct} is satisfied;
\item \ref{growthcond} is satisfied if and only if $r_1<p$ and $s_2<q$, as well as $p'<\frac{q}{r_2-1}$ and $q'<\frac{p}{s_1-1}$;
\item \ref{degiorgicond} is satisfied if and only if $\max\{(N-q)(r_2-1),(N-p)(s_1-1)\}<pq$;
\item \ref{harnackcond} is satisfied if and only if $(1-\alpha)(1-\beta)<(p-1)(q-1)$;
\item \ref{weightscond} is satisfied;
\item \ref{decaycond} is satisfied if and only if $\theta>N+\max\left\{\alpha\frac{N-p}{p-1},\beta\frac{N-q}{q-1}\right\}$;
\item \ref{diazsaacond} is satisfied if and only if $r_1<p$ and $s_2<q$.
\end{itemize}
In particular, in order to verify \ref{growthcond}, apply Young's inequality to the terms $u^{\hat{r}-1}v^{\hat{s}}$ and $u^{\hat{r}}v^{\hat{s}-1}$, using \eqref{youngcond}.
\end{ex}
For instance, all the hypotheses are verified provided $N\geq 3$, $0<\alpha<\hat{r}=1<r_1<p<N$, and $0<\beta<\hat{s}=1<s_2<q<N$, by taking $r_2=s_1=2$, with either:
\begin{itemize}
\item $p,q>\max\left\{\frac{\sqrt{1+4N}-1}{2},2\right\}$ and $\theta>2(N-1)$;
\item $p>N-1$ and $q>\frac{N-1}{N-2}$, as well as $\theta>(N-1)^2$.
\end{itemize}

\begin{ex}[The non-power case]
Consider $\Phi(t):=t^p \log^{\hat{\sigma}}(1+t)$ and $\Psi(t):=t^q \log^{\hat{\tau}}(1+t)$, with $\hat{\sigma},\hat{\tau}\in[0,1)$, $N\geq 2$, $p\in(1,N-\hat{\sigma})$, $q\in(1,N-\hat{\tau})$, $q\geq p'$ (e.g., $p,q\geq 2$) and $\min\{p+qq',q+pp'\}>N$ (e.g., $p,q>N-4$). Moreover, for $\sigma\in[0,\hat{\sigma})$ and $\tau\in[0,\hat{\tau})$, define
\begin{equation*}
\begin{alignedat}{2}
&\Upsilon_1:=t^p \log^\sigma(1+t), \quad &&\Upsilon_2:=\int_0^t y_2(\tau) \dtau, \\
&\Gamma_1:=\int_0^t \gamma_1(\tau) \dtau, \quad &&\Gamma_2:=t^q \log^\tau(1+t),
\end{alignedat}
\end{equation*}
where $y_2,\gamma_1\in C^1(\R_+)$ are chosen such that
\begin{equation*}
y_2(t):=\left\{
\begin{alignedat}{2}
&t, \quad &&\mbox{if} \;\; t\in(0,1), \\
&\mbox{increasing}, \quad &&\mbox{if} \;\; t\in(1,2), \\
&t^{\frac{q}{p'}}\log^{\frac{\hat{\tau}}{p'}}(1+t), \quad &&\mbox{if} \;\; t\in(2,+\infty),
\end{alignedat}
\right.
\quad
\gamma_1(t):=\left\{
\begin{alignedat}{2}
&t, \quad &&\mbox{if} \;\; t\in(0,1), \\
&\mbox{increasing}, \quad &&\mbox{if} \;\; t\in(1,2), \\
&t^{\frac{p}{q'}}\log^{\frac{\hat{\sigma}}{q'}}(1+t), \quad &&\mbox{if} \;\; t\in(2,+\infty),
\end{alignedat}
\right.
\end{equation*}
Take
\begin{equation*}
\begin{split}
f(u,v) := (1-\alpha)u^{-\alpha}\Upsilon_2'(v) + \Upsilon_1'(u) + \Gamma_1''(u)v^{1-\beta}, \\
g(u,v) := (1-\beta)\Gamma_1'(u)v^{-\beta} + \Gamma_2'(v) + u^{1-\alpha}\Upsilon_2''(v), \\
\end{split}
\end{equation*}
being $\alpha\in(\frac{1}{q},1)$ and $\beta\in(\frac{1}{p},1)$. Finally, set $h$ and $k$ to be equal to $w:=\frac{1}{1+|x|^\theta}$ with $\theta>N+\max\left\{\alpha\frac{N-p}{p-1},\beta\frac{N-q}{q-1}\right\}$. Then all the hypotheses except \ref{diazsaacond} are satisfied with $\nu_1=\nu_2=1$ and
\begin{equation*}
H(x,u,v) := w(x)\left[ \Upsilon_1(u) + \Gamma_2(v) + u^{1-\alpha}\Upsilon_2'(v) + \Gamma_1'(u)v^{1-\beta} \right],
\end{equation*}
provided $\hat{\sigma}<\min\left\{\frac{\beta p-1}{q-1},\frac{p^2}{N-p},\frac{p(qq'-N+p)}{N-p}\right\}$ and $\hat{\tau}<\min\left\{\frac{\alpha q-1}{p-1},\frac{q^2}{N-q},\frac{q(pp'-N+q)}{N-q}\right\}$. Anyway, these smallness conditions can be removed by using, for all the Young functions involved, the indices at infinity \cite[pp.26-27]{RR} instead of the ones defined in \eqref{indices}.

It is worth observing that, if $\sigma,\tau>0$, the functions $\Upsilon_i,\Gamma_i$, $i=1,2$, cannot be replaced with power functions, because of \ref{growthcond}: indeed,
\begin{equation*}
\Upsilon_1(t) = t^p \log^\sigma(1+t) < t^r < t^p \log^{\hat{\sigma}}(1+t) = \Phi(t) \quad \mbox{is false for all} \; r>1,
\end{equation*}
as well as (cf. \eqref{prodcond})
\begin{equation*}
(\overline{\Upsilon_2}^{-1} \circ \Upsilon_2)(t) \simeq y_2(t) = t^{\frac{q}{p'}}\log^{\frac{\hat{\tau}}{p'}}(1+t) < t^r < (\overline{\Phi}^{-1}\circ\Psi)(t) \quad \mbox{is false for all} \; r>1,
\end{equation*}
and the same holds for $\Gamma_2,\Upsilon_1$.

To guarantee also \ref{diazsaacond}, it suffices to consider
\begin{equation*}
\begin{split}
f(u,v) := (1-\alpha)u^{-\alpha}\Upsilon_2'(v) + \Gamma_1''(u)v^{1-\beta}, \\
g(u,v) := (1-\beta)\Gamma_1'(u)v^{-\beta} + u^{1-\alpha}\Upsilon_2''(v), \\
\end{split}
\end{equation*}
choosing
\begin{equation*}
\begin{split}
\Upsilon_1(t):=t^r, \quad \Upsilon_2(t):=t^{\tilde{r}+2}\log^{\tau}(1+t), \\
\Gamma_1(t):=t^{\tilde{s}+2}\log^{\sigma}(1+t), \quad \Gamma_2(t):=t^s, \\
\end{split}
\end{equation*}
with $p,q>2$, $\max\{N-p,N-q\}<pq$, and $\tilde{r},\tilde{s},\sigma,\tau>0$ small enough, as well as $1<r<p$, $1<s<q$ such that
\begin{equation*}
\frac{\tilde{r}+\sigma}{r-1} + \frac{1-\beta}{\tilde{s}+\tau+1} \leq 1, \quad \frac{1-\alpha}{\tilde{r}+\sigma+1} + \frac{\tilde{s}+\tau}{s-1} \leq 1,
\end{equation*}
which parallels \eqref{youngcond} of the power case (see Example \ref{powercase}). The smallness conditions for $\tilde{r},\tilde{s},\sigma,\tau$ can be derived as in Example \ref{powercase}, using the indices $s_{\Upsilon_2}$ and $s_{\Gamma_1}$ (see \eqref{indices}).
\end{ex}

\section{Preliminaries}

\subsection{Notations}
Given $r>1$, we set $r':=\frac{r}{r-1}$ and, provided $r<N$, $r^*:=\frac{Nr}{N-r}$, called respectively Young and Sobolev conjugates of $r$. The symbol $\R_+$ stands for $(0,+\infty)$, while $B_r^e$ indicates the set $\R^N\setminus \overline{B}_r$, where $\overline{B}_r$ is the closure of the open ball $B_r$ centered at the origin and of radius $r$. The characteristic function of a set $A$ will be denoted by $\chi_A$; if $A\subseteq \R^N$, we write $|A|$ for the $N$-dimensional Lebesgue measure of $A$. We indicate with $f\circ g$ the composition between the functions $f$ and $g$, that is, $(f\circ g)(t):=f(g(t))$ for all $t$. If $f,g:\R^N\to\R^M$, $M\geq 1$, then $f*g$ stands for the convolution between $f$ and $g$, i.e., $(f*g)(x):=\int_{\R^N} f(y)g(x-y)\dy$ for all $x\in\R^N$.

Given $u: \R^N \to \R$ and $k \in \R$, we define its positive and negative parts as $u_+$ and $u_-$. The set $\Omega_k:=\{x\in\R^N: \, u(x)>k\}$ is the super-level set of $u$ at level $k$, when the dependence on $u$ is clear. The symbol $u^{**}$ stands for the average of $u^*$, that is, $u^{**}(t) := \frac{1}{t}\int_0^t u^*(r) \dr$ (see Definition \ref{rearrangdef}).

The space $C^\infty_c(\R^N)$ is the set of the compactly supported test functions, endowed with its standard topology, and $\{\rho_n\}\subseteq C^\infty_c(\R^N)$ denotes a sequence of standard mollifiers (vide, e.g., \cite[pp.108-111]{B}). We write $z\in Z_{\rm loc}(\R^N)$ if for every nonempty compact subset $ K $ of $\R^N$ the restriction $z_{\mid_K}$ belongs to $Z(K)$. Similarly, a sequence $\{z_n\}\subseteq Z_{\rm loc}(\R^N)$ is called bounded in $ Z_{\rm loc}(\R^N) $ once the same holds for $ \{z_{n\mid_K}\} $ in $ Z(K) $, with any $K$ as above. Whenever the domain of integration is clear, for any $p\in[1,+\infty]$ and Young function $\Lambda$ the norms $\|\cdot\|_p$ and $\|\cdot\|_\Lambda$ stand for the usual $L^p$ and $L^\Lambda$ norms, respectively.

In the whole paper, the symbol $C$ will denote a positive constant which may change its value at each passage. Dependencies of $C$ will be indicated from time to time, and some of them will be emphasized using subscripts (for instance, $C_\sigma$ depends on various quantities, and in particular on $\sigma$). To avoid technicalities, we write `for all $x$' instead of `for almost all $x$'.

\subsection{Rearrangements, Young functions, and Orlicz spaces}

\begin{defi}
\label{rearrangdef}
Let $f:\R^N\to\R$ be a measurable function. The non-increasing rearrangement of $f$ is $f^*:\R_+\to[0,+\infty]$ defined as
\begin{equation*}
f^*(r) := \inf\{t>0: \mu_f(t)\leq r\} \quad \forall r\geq 0,
\end{equation*}
where $\mu_f:[0,+\infty)\to[0,+\infty)$ is the distribution function of $f$, that is,
\begin{equation*}
\mu_f(t) := |\{x\in\R^N: \, |f(x)|>t\}| \quad \forall t\geq 0.
\end{equation*}
\end{defi}

In view of \cite[Lemmas 1.8.10 and 1.8.12]{Zi}, for all $p\in (1,+\infty)$ we have
\begin{equation}
\label{birearrang}
f^{**}(t) \leq p' t^{-\frac{1}{p}} \|f\|_p \quad \forall t\geq 0.
\end{equation}

\begin{defi}
Let $\Lambda_1,\Lambda_2$ be two functions. We write $\Lambda_1 < \Lambda_2$ if there exist $c,T>0$ such that
\begin{equation*}
\Lambda_1(t) \leq \Lambda_2(ct) \quad \forall t \geq T.
\end{equation*}
We write $\Lambda_1 \ll \Lambda_2$ if 
\begin{equation*}
\lim_{t \to +\infty} \frac{\Lambda_1(t)}{\Lambda_2(\eta t)} = 0 \quad \forall \eta > 0.
\end{equation*}
\end{defi}
We recall that
\begin{equation*}
s_{\Psi_1} < i_{\Psi_2} \quad \Rightarrow \quad \Psi_1 \ll \Psi_2 \quad \Rightarrow \quad \Psi_1 < \Psi_2,
\end{equation*}
and the reverse implications are generally false.

\begin{defi}
A function $\Lambda:[0,+\infty) \to [0,+\infty)$ is called Young function if it is convex, $\Lambda(t)=0$ if and only if $t=0$, and the following holds true:
\begin{equation}
\label{Nfunct}
\lim_{t \to 0^+} \frac{\Lambda(t)}{t} = 0, \quad \lim_{t \to +\infty} \frac{\Lambda(t)}{t} = +\infty.
\end{equation}
\end{defi}

\begin{defi}
\label{conj}
Let $\Lambda$ be a Young function. We denote by $\overline{\Lambda}$ the Young conjugate of $\Lambda$, defined as
\begin{equation*}
\overline{\Lambda}(t) := \max_{s \geq 0} \{st-\Lambda(s)\} \quad \forall t \geq 0.
\end{equation*}
\end{defi}

Incidentally, we recall (cf. \cite[p.10]{RR}) that
\begin{equation}
\label{invder}
\overline{\Lambda}' = (\Lambda')^{-1}
\end{equation}
for any Young function $\Lambda$.

The function $\Lambda_*$ in the following definition was introduced by Cianchi in \cite{C} and, in an equivalent form, in \cite{C0}.

\begin{defi}
Let $\Lambda$ be a Young function satisfying
\begin{equation}
\label{cianchicond}
\int_1^{+\infty} \left(\frac{t}{\Lambda(t)}\right)^{N'-1} \dt = +\infty.
\end{equation}
We indicate with $\Lambda_*$ the Sobolev conjugate of $\Lambda$, defined as $\Lambda_*:=\Lambda\circ \mathscr{H}^{-1}$, where $\mathscr{H}$ is
\begin{equation*}
\mathscr{H}(t) := \left(\int_0^t \left(\frac{\tau}{\Lambda(\tau)}\right)^{N'-1} \dtau\right)^{\frac{1}{N'}}.
\end{equation*}
\end{defi}

\begin{defi}
\label{delta2}
Let $\Lambda$ be a Young function. We write $\Lambda \in \Delta_2$ if there exist $k,T>0$ such that
\begin{equation*}
\Lambda(2t) \leq k\Lambda(t) \quad \forall t \geq T.
\end{equation*}
We write $\Lambda \in \nabla_2$ if there exist $\eta>1$ and $T>0$ such that
\begin{equation*}
\Lambda(t) \leq \frac{1}{2\eta}\Psi(\eta t) \quad \forall t \geq T.
\end{equation*}
\end{defi}

Let $\Lambda \in \Delta_2$. Then the Orlicz class
\begin{equation*}
L^\Lambda(\R^N) := \left\{ u:\R^N \to \R \; \mbox{measurable:} \, \int_{\R^N} \Lambda(|u(x)|) \dx < +\infty \right\}
\end{equation*}
becomes a Banach spaces when equipped with the Luxembourg norm
\begin{equation*}
\|u\|_{L^\Lambda(\R^N)} := \inf \left\{ \tau>0: \, \int_{\R^N} \Lambda\left(\frac{|u(x)|}{\tau}\right) \dx \leq 1 \right\}.
\end{equation*}
Analogously (cf. \cite{V}), we will consider the weighted Orlicz spaces
\begin{equation*}
L^\Lambda(\R^N;w) := \left\{ u:\R^N \to \R \; \mbox{measurable:} \, \int_{\R^N} w\Lambda(|u(x)|) \dx < +\infty \right\}
\end{equation*}
equipped with the Luxembourg norm
\begin{equation*}
\|u\|_{L^\Lambda(\R^N;w)} := \inf \left\{ \tau>0: \, \int_{\R^N} w\Lambda\left(\frac{|u(x)|}{\tau}\right) \dx \leq 1 \right\}.
\end{equation*}

Suppose that $\Lambda$ is a Young function whose indices $i_\Lambda,s_\Lambda$ obey
\begin{equation}
\label{indices}
1 < i_\Lambda := \inf_{t>0} \frac{t\Lambda'(t)}{\Lambda(t)} \leq \sup_{t>0} \frac{t\Lambda'(t)}{\Lambda(t)} =: s_\Lambda < +\infty,
\end{equation}
which implies $\Lambda \in \Delta_2 \cap \nabla_2$. We define the functions $\underline{\zeta}_\Lambda,\overline{\zeta}_\Lambda:[0,+\infty) \to [0,+\infty)$ as
\begin{equation*}
\underline{\zeta}_\Lambda(t) := \min\{t^{i_\Lambda},t^{s_\Lambda}\}, \quad \overline{\zeta}_\Lambda(t) := \max\{t^{i_\Lambda},t^{s_\Lambda}\}.
\end{equation*}
One has (cf. \cite[Lemma 2.1]{FIN})
\begin{equation}
\label{factor}
\underline{\zeta}_\Lambda(k) \Lambda(t) \leq \Lambda(kt) \leq \overline{\zeta}_\Lambda(k) \Lambda(t) \quad \forall k,t \geq 0
\end{equation}
and
\begin{equation*}
\begin{split}
&\underline{\zeta}_\Lambda(\|w\|_{L^\Lambda(\R^N)}) \leq \int_{\R^N} \Lambda(|w(x)|) \dx \leq \overline{\zeta}_\Lambda(\|w\|_{L^\Lambda(\R^N)})
\end{split}
\end{equation*}
for all $w \in L^\Lambda(\R^N)$. We also recall (see \cite[Lemmas 2.4-2.5]{FIN}) that
\begin{equation}
\label{youngind}
s_\Lambda' \leq i_{\overline{\Lambda}} \leq s_{\overline{\Lambda}} \leq i_\Lambda'
\end{equation}
and, provided $s_\Lambda<N$,
\begin{equation}
\label{sobind}
i_\Lambda^* \leq i_{\Lambda_*} \leq s_{\Lambda_*} \leq s_\Lambda^*.
\end{equation}

In the sequel we will use the following inequalities, that are consequence of \eqref{conj}, \eqref{delta2}, and \eqref{indices}; cf. \cite{RR}. Here $\Lambda,\Lambda_1,\Lambda_2$ are Young functions ($\Lambda_1(t)=t$ is allowed).
\begin{equation}
\label{strictdom}
\Lambda_1(t) \leq \sigma \Lambda_2(t) + C_\sigma \quad \forall t,\sigma>0, \quad \mbox{provided} \;\; \Lambda_1\ll\Lambda_2.
\end{equation}
\begin{equation}
\label{fundineq0}
t \leq \Lambda^{-1}(t) \overline{\Lambda}^{-1}(t) \leq 2t \quad \forall t>0.
\end{equation}
\begin{equation}
\label{fundineq}
\Lambda(t) \leq t\overline{\Lambda}^{-1}(\Lambda(t)) \leq 2\Lambda(t) \quad \forall t>0.
\end{equation}
\begin{equation}
\label{fundineq2}
\overline{\Lambda}\left(\frac{\Lambda(t)}{t}\right) \leq \Lambda(t) \quad \forall t>0.
\end{equation}
\begin{equation}
\label{subadd}
\Lambda(s+t) \leq C(\Lambda(s)+\Lambda(t)) \quad \forall s,t>0.
\end{equation}
\begin{equation}
\label{invgrowth}
\Lambda^{-1}(t) \leq C t^{\frac{1}{i_\Lambda}} \quad \forall t>1.
\end{equation}

Suppose \eqref{cianchicond}. We introduce the Beppo Levi-Orlicz spaces $\mathcal{D}^{1,\Lambda}_0(\R^N)$ as the closure of $C^\infty_c(\R^N)$ under the norm
\begin{equation*}
\|u\|_{\mathcal{D}^{1,\Lambda}_0(\R^N)} := \|\nabla u\|_{L^\Lambda(\R^N)}.
\end{equation*}
Equivalently,
\begin{equation*}
\mathcal{D}^{1,\Lambda}_0(\R^N) = \{u\in L^{\Lambda_*}(\R^N): \, |\nabla u|\in L^\Lambda(\R^N)\}.
\end{equation*}
Indeed, the following continuous embedding holds true (vide \cite[p.1634]{C}):
\begin{equation}
\label{embedding}
\mathcal{D}^{1,\Lambda}_0(\R^N) \hookrightarrow L^{\Lambda_*}(\R^N).
\end{equation}

We consider the following function spaces:
\begin{equation*}
X := \mathcal{D}^{1,\Phi}_0(\R^N) \times \mathcal{D}^{1,\Psi}_0(\R^N),
\end{equation*}
\begin{equation*}
Y := L^{\Phi_*}(\R^N) \times L^{\Psi_*}(\R^N).
\end{equation*}

\begin{defi}
A couple $(u,v)\in X$, $u,v>0$ in $\R^N$, is called weak solution to 
\begin{equation*}
\left\{
\begin{alignedat}{2}
-\Delta_{\Phi} u &= h(x)f(u,v) &&\quad \mbox{in}\;\; \R^N, \\
-\Delta_{\Psi} v &= k(x)g(u,v) &&\quad \mbox{in}\;\; \R^N, \\
\end{alignedat}
\right.
\end{equation*}
if,
for any $(\xi,\nu)\in X$,
\begin{equation}
\label{defsol}
\begin{split}
\int_{\R^N} \frac{\phi(|\nabla u|)}{|\nabla u|} \nabla u \nabla \xi \dx = \int_{\R^N} h(x)f(u,v) \xi \dx, \\
\int_{\R^N} \frac{\psi(|\nabla v|)}{|\nabla v|} \nabla v \nabla \nu \dx = \int_{\R^N} k(x)g(u,v) \nu \dx. \\
\end{split}
\end{equation}
\end{defi}

\begin{defi}
A couple $(u,v)\in X$ is called distributional solution to \eqref{prob} if \eqref{defsol} holds true for any $(\xi,\nu)\in C^\infty_c(\R^N)\times C^\infty_c(\R^N)$.
\end{defi}


\subsection{Useful propositions}
\begin{prop}[Pratt's lemma]
\label{pratt}
Let $ (\Omega,\mu,\mathscr{F}) $ be a measure space. Suppose $ \{f_n\} $, $ \{g_n\} $, $ \{h_n\} $ to be sequences of measurable functions such that
\begin{equation*}
\begin{alignedat}{2}
&f_n \to f \quad \mu\mbox{-a.e. in} \;\; \Omega, \quad &&\lim_{n \to \infty} \int_\Omega f_n \, {\rm d}\mu = \int_\Omega f \, {\rm d}\mu \in \R, \\
&h_n \to h \quad \mu\mbox{-a.e. in} \;\; \Omega, \quad \quad &&\lim_{n \to \infty} \int_\Omega h_n \, {\rm d}\mu = \int_\Omega h \, {\rm d}\mu \in \R, \\
&g_n \to g \quad \mu\mbox{-a.e. in} \;\; \Omega, \quad &&f_n \leq g_n \leq h_n \quad \mu\mbox{-a.e. in} \;\; \Omega.
\end{alignedat}
\end{equation*}
Then
\begin{equation*}
\lim_{n \to \infty} \int_\Omega g_n \, {\rm d}\mu = \int_\Omega g \, {\rm d}\mu \in \R.
\end{equation*}
\end{prop}
\begin{proof}
It suffices to apply Fatou's lemma to both $ g_n-f_n $ and $ h_n-g_n $.
\end{proof}


\begin{lemma}
\label{intind}
Let $\Lambda$ be a Young function of class $C^2$. Set $\lambda:=\Lambda'$ and suppose that
\begin{equation*}
0 < i_\lambda := \inf_{t>0}\frac{t\lambda'(t)}{\lambda(t)} \leq \sup_{t>0}\frac{t\lambda'(t)}{\lambda(t)} =: s_\lambda < +\infty.
\end{equation*}
Then
\begin{equation*}
i_\lambda+1\leq i_\Lambda\leq s_\Lambda\leq s_\lambda+1.
\end{equation*}
\end{lemma}

\begin{proof}
Observe that
\begin{equation*}
(i_\lambda+1)\lambda(s) \leq \frac{{\rm d}}{{\rm d}s}(s\lambda(s)) \leq (s_\lambda+1)\lambda(s) \quad \forall s>0.
\end{equation*}
Integrating in $(0,t)$ and dividing by $\Lambda(t)$ produces
\begin{equation*}
i_\lambda+1 \leq \frac{t\Lambda'(t)}{\Lambda(t)} \leq s_\lambda+1 \quad \forall t>0,
\end{equation*}
yielding the conclusion.
\end{proof}

\begin{prop}
\label{convex}
Let $\Lambda,\Theta$ be two Young functions of class $C^2$. Denoting by
\begin{equation*}
\begin{split}
&i_\lambda := \inf_{t>0} \frac{t\Lambda''(t)}{\Lambda'(t)}, \quad s_\lambda := \sup_{t>0} \frac{t\Lambda''(t)}{\Lambda'(t)}, \\
&i_\theta := \inf_{t>0} \frac{t\Theta''(t)}{\Theta'(t)}, \quad s_\theta := \sup_{t>0} \frac{t\Theta''(t)}{\Theta'(t)},
\end{split}
\end{equation*}
suppose that
\begin{equation*}
0 < i_\lambda \leq s_\lambda \leq i_\theta \leq s_\theta < +\infty.
\end{equation*}
Then $\Upsilon:=\Theta\circ\Lambda^{-1}$ is convex. If, moreover, $s_\Lambda<i_\Theta$, then $\Upsilon$ is a Young function.
\end{prop}

\begin{proof}
First we notice that $\Upsilon$ is of class $C^2$. We have
\begin{equation}
\label{composition}
\Upsilon' = (\Theta'\circ\Lambda^{-1}) \cdot (\Lambda^{-1})' = (\Theta'\circ \Lambda^{-1}) \cdot \left( \frac{1}{\Lambda'}\circ\Lambda^{-1} \right) = \frac{\Theta'}{\Lambda'}\circ\Lambda^{-1},
\end{equation}
so $\Upsilon$ is strictly increasing, since $\Theta',\Lambda'>0$ in $\R_+$. Differentiating we get
\begin{equation*}
\begin{split}
\Upsilon'' &= \left( \frac{\Theta'}{\Lambda'}\circ\Lambda^{-1} \right)' = \left[ \left( \frac{\Theta'}{\Lambda'} \left( \frac{\Theta''}{\Theta'}-\frac{\Lambda''}{\Lambda'} \right) \right) \circ \Lambda^{-1} \right] \cdot (\Lambda^{-1})' \\
&=\left[ \left( \frac{\Theta'}{\Lambda'} \left( \frac{\Theta''}{\Theta'}-\frac{\Lambda''}{\Lambda'} \right) \right) \circ \Lambda^{-1} \right] \cdot \left(\frac{1}{\Lambda'}\circ\Lambda^{-1}\right) = \left[ \frac{\Theta'}{(\Lambda')^2} \left( \frac{\Theta''}{\Theta'}-\frac{\Lambda''}{\Lambda'} \right) \right] \circ\Lambda^{-1}.
\end{split}
\end{equation*}
Moreover,
\begin{equation*}
\frac{\Theta''(t)}{\Theta'(t)}-\frac{\Lambda''(t)}{\Lambda'(t)} = \frac{1}{t} \left(\frac{t\Theta''(t)}{\Theta'(t)}-\frac{t\Lambda''(t)}{\Lambda'(t)}\right) \geq \frac{1}{t}(i_\theta-s_\lambda)\geq 0\quad \forall t>0.
\end{equation*}
Thus $\Upsilon''\geq 0$, ensuring the convexity of $\Upsilon$.

Now suppose $s_\Lambda<i_\Theta$. By \eqref{composition} and the definition of $\Upsilon$ we deduce that
\begin{equation*}
\frac{t\Upsilon'(t)}{\Upsilon(t)} = \frac{t\Theta'(\Lambda^{-1}(t))}{\Theta(\Lambda^{-1}(t))\Lambda'(\Lambda^{-1}(t))} = \frac{	\Lambda(\tau)\Theta'(\tau)}{\Theta(\tau)\Lambda'(\tau)} = \frac{\frac{\tau\Theta'(\tau)}{\Theta(\tau)}}{\frac{\tau\Lambda'(\tau)}{\Lambda(\tau)}},
\end{equation*}
where $\tau:=\Lambda^{-1}(t)$. Accordingly, besides Lemma \ref{intind}, we infer
\begin{equation}
\label{nonyoungcompind}
1 < \frac{i_\Theta}{s_\Lambda} \leq i_\Upsilon \leq s_\Upsilon \leq \frac{s_\Theta}{i_\Lambda} \leq \frac{s_\theta+1}{i_\lambda+1} < +\infty,
\end{equation}
ensuring \eqref{Nfunct} for $\Upsilon$, which is therefore a Young function.
\end{proof}

The following lemma can be found in \cite[Lemma VI.3.2]{RR}.

\begin{lemma}
\label{RRgen}
Let $\Phi$, $\Psi$ be two Young functions such that $\Phi \ll \Psi$. Then there exists a Young function $\Lambda$ such that $\Phi\ll\Lambda\ll\Psi$. Such $\Lambda$ is called intermediate function between $\Phi$ and $\Psi$.
\end{lemma}

\begin{prop}
\label{lambdadef}
Under \eqref{prodcond}, there exist two Young functions $\Lambda_1$, $\Lambda_2$ such that
\begin{equation*}
\label{prodcond2}
\overline{\Phi} \ll \Lambda_1 \ll \Psi\circ\Upsilon_2^{-1}\circ\overline{\Upsilon}_2 \quad \mbox{and} \quad \overline{\Psi} \ll \Lambda_2 \ll \Phi\circ\Gamma_1^{-1}\circ\overline{\Gamma}_1.
\end{equation*} 
\end{prop}
\begin{proof}
Set $\Theta:=\Psi\circ\Upsilon_2^{-1}\circ\overline{\Upsilon}_2$. A careful inspection of \cite[Theorem II.2.2(b)]{RR} reveals that there exists a Young function $R$ such that $\Xi := \overline{\Phi}\circ R < \Theta$ even if $\Theta$ is not a Young function. Thus, using \eqref{factor} and $\Xi < \Theta$, we have
\begin{equation*}
\Xi(t) \leq \Xi(1)t^{i_\Xi}\chi_{\{t\leq 1\}} + M\Theta(ct) \chi_{\{t>1\}} =:\hat{\Theta}(t) \quad \forall t>0,
\end{equation*}
for suitable $M,c>0$. Let us consider $F^{\star\star}:=(F^\star)^\star$ to be the convex bi-conjugate of a generic function $F$ (see, e.g., \cite[Definition 51.1]{Z/III} for the definition of convex conjugate). From \cite[Proposition 51.6]{Z/III} we have $\hat{\Theta}^{\star\star}\leq \hat{\Theta}$ and $\Xi^{\star\star}=\Xi$, since $\Xi$ is convex and continuous. Moreover, the convex conjugation is a decreasing operator (vide \cite[Proposition 51.6]{Z/III}), so that $\Xi\leq \hat{\Theta}$ implies $\hat{\Theta}^\star \leq \Xi^\star$, which in turn gives $\Xi^{\star\star} \leq \hat{\Theta}^{\star\star}$. Then
\begin{equation}
\label{biconj}
\Xi=\Xi^{\star\star} \leq \hat{\Theta}^{\star\star} \leq \hat{\Theta} \quad \mbox{in} \;\; (0,+\infty).
\end{equation}
Define $\hat{\Lambda}_1:=\frac{1}{2M} \hat{\Theta}^{\star\star}$. Convexity of $\hat{\Lambda}_1$ is guaranteed by \cite[Proposition 51.6]{Z/III}; moreover, it is readily seen that \eqref{biconj} ensures \eqref{Nfunct} for $\Lambda:=\hat{\Lambda}_1$. Hence $\hat{\Lambda}_1$ is a Young function. Obviously, $\hat{\Lambda}_1 < \Theta$. On the other hand, for any fixed $\eps,C>0$, we choose $T>0$ such that
\begin{equation}
\label{superlin2}
\frac{R(t)}{t} > \frac{1}{\eps} \max\{2MC,1\} \quad \forall t>T,
\end{equation}
which is possible by super-linearity of $R$. Thus, recalling \eqref{biconj}, the convexity of $\overline{\Phi}$, and \eqref{superlin2},
\begin{equation*}
\frac{\hat{\Lambda}_1(\eps t)}{\overline{\Phi}(t)} = \frac{\hat{\Lambda}_1(\eps t)}{\Xi(\eps t)} \frac{\Xi(\eps t)}{\overline{\Phi}(t)} \geq \frac{1}{2M} \frac{\overline{\Phi}(R(\eps t))}{\overline{\Phi}(t)} = \frac{1}{2M} \frac{\overline{\Phi}\left(\frac{R(\eps t)}{t} t\right)}{\overline{\Phi}(t)} \geq \frac{\eps}{2M} \inf_{t>\frac{T}{\eps}} \frac{R(\eps t)}{\eps t} \geq C
\end{equation*}
holds true for any $t>\frac{T}{\eps}$. Arbitrariness of $C$ and $\eps$ permits to conclude $\overline{\Phi} \ll \hat{\Lambda}_1$. Invoking Proposition \ref{RRgen} produces a Young function $\Lambda_1$ such that $\overline{\Phi}\ll\Lambda_1\ll\hat{\Lambda}_1<\Theta$. The same argument furnishes $\Lambda_2$ with the required properties.
\end{proof}

\begin{prop}
\label{interpolation}
Let $h$ belong to $L^1(\R^N) \cap L^\infty(\R^N)$. Then $h \in L^{\Lambda}(\R^N)$ for all Young functions $\Lambda$.
\end{prop}
\begin{proof}
According to \eqref{Nfunct}, there exists $\sigma>0$ such that $\Lambda(t) < t$ for all $t\in(0,\sigma]$. Set $\tau := \max\{\|h\|_1,\sigma^{-1}\|h\|_\infty\}$. Since $|\{|h|>\tau\sigma\}| = 0$, we have
\begin{equation*}
\begin{split}
\int_{\R^N} \Lambda \left( \frac{|h|}{\tau} \right) \dx &= \int_{\{|h|\leq\tau\sigma\}} \Lambda \left( \frac{|h|}{\tau} \right) \dx \leq \frac{1}{\tau} \int_{\{|h|\leq\tau\sigma\}} |h| \dx \leq \frac{\|h\|_1}{\tau} \leq 1.
\end{split}
\end{equation*}
Thus $\|h\|_\Lambda \leq \max\{\|h\|_1,\sigma^{-1}\|h\|_\infty\}$.
\end{proof}

\begin{prop}
\label{interpolation2}
Let $\Lambda,\Theta$ be two Young functions such that $\Lambda<\Theta$, and let $h\in L^1(\R^N) \cap L^\infty(\R^N)$. Then there exists $C=C(\Lambda,\Theta,\|h\|_1,\|h\|_\infty)>0$ such that $\|hu\|_\Lambda \leq C \|u\|_\Theta$ for all $u\in L^\Theta(\R^N)$.
\end{prop}
\begin{proof}
By $\Lambda<\Theta$ we have, for suitable $c,T>0$,
\begin{equation*}
\Lambda(t) < \Theta(ct) \quad \forall t\geq T.
\end{equation*}
Set $\tau:=cR\|h\|_\infty\|u\|_\Theta$, with $R\geq 1$ to be chosen. We compute
\begin{equation*}
\begin{split}
\int_{\R^N} \Lambda\left(\frac{|hu|}{\tau}\right) \dx &= \int_{\R^N} \Lambda\left(\frac{|hu|}{cR\|h\|_\infty\|u\|_\Theta}\right) \dx \leq \frac{1}{R} \int_{\R^N} \frac{|h|}{\|h\|_\infty} \Lambda\left(\frac{|u|}{c\|u\|_\Theta}\right) \dx \\
&\leq \frac{1}{R} \left[ \int_{\{|u|\geq cT\|u\|_\Theta\}} \Theta\left(\frac{|u|}{\|u\|_\Theta}\right) \dx + \frac{\Lambda(T)}{\|h\|_\infty} \int_{\{|u|<cT\|u\|_\Theta\}} |h| \dx \right] \\
&\leq \frac{1}{R} \left(1+\Lambda(T)\frac{\|h\|_1}{\|h\|_\infty}\right).
\end{split}
\end{equation*}
Choosing $R:=1+\Lambda(T)\frac{\|h\|_1}{\|h\|_\infty}$ we deduce
\begin{equation*}
\|hu\|_\Lambda \leq \tau = c\left(1+\Lambda(T)\frac{\|h\|_1}{\|h\|_\infty}\right)\|h\|_\infty\|u\|_\Theta = c(\|h\|_\infty+\Lambda(T)\|h\|_1)\|u\|_\Theta.
\end{equation*}
\end{proof}

\begin{lemma}[Young's inequality]
\label{youngineq}
Let $\Lambda_i,\Theta$, $i=1,2,3$, be four Young functions with $\Theta\in\Delta_2$. If
\begin{equation*}
\Lambda_1^{-1}(t)\Lambda_2^{-1}(t)\Lambda_3^{-1}(t) \leq k\Theta^{-1}(t) \quad \forall t>0
\end{equation*}
for some $k>0$, then
\begin{equation*}
\Theta(xyz) \leq C[\Lambda_1(x)+\Lambda_2(y)+\Lambda_3(z)] \quad \forall x,y,z>0
\end{equation*}
for some $C>0$ depending on $k,\Theta$.
\end{lemma}
\begin{proof}
The proof is patterned after that of \cite[Lemma 2.1]{ON}. Suppose, without loss of generality, that $\Lambda_1(x)\leq \Lambda_2(y)\leq \Lambda_3(z)$. Then, by hypothesis,
\begin{equation*}
xyz = \Lambda_1^{-1}(\Lambda_1(x)) \Lambda_2^{-1}(\Lambda_2(y)) \Lambda_3^{-1}(\Lambda_3(z)) \leq \Lambda_1^{-1}(\Lambda_3(z)) \Lambda_2^{-1}(\Lambda_3(z)) \Lambda_3^{-1}(\Lambda_3(z)) \leq k\Theta^{-1}(\Lambda_3(z)).
\end{equation*}
Thus, since $\Theta\in\Delta_2$, there exists $C=C(k,\Theta)>0$ such that
\begin{equation*}
\Theta(xyz) \leq \Theta(k\Theta^{-1}(\Lambda_3(z))) \leq C\Theta(\Theta^{-1}(\Lambda_3(z))) = C\Lambda_3(z) \leq C[\Lambda_1(x)+\Lambda_2(y)+\Lambda_3(z)].
\end{equation*}
Exchanging the role of $\Lambda_1$,$\Lambda_2$,$\Lambda_3$ concludes the proof.
\end{proof}

\begin{prop}
\label{holderineq}
Let $\Lambda$ be a Young function satisfying $s_\Lambda<N$. Then, for any $u\in L^{\Lambda_*}(\R^N)$ and $v\in L^N(\R^N)$, one has $\|uv\|_\Lambda \leq C \|u\|_{\Lambda_*} \|v\|_N$.
\end{prop}

\begin{proof}
According to H\"older's inequality \cite[Theorem 6.7]{ON}, it is sufficient to prove
\begin{equation}
\label{Holdcond}
t^{\frac{1}{N}} \Lambda_*^{-1}(t) \leq C \Lambda^{-1}(t) \quad \forall t>0,
\end{equation}
which is equivalent (setting $s:=\Lambda^{-1}(t)$) to
\begin{equation}
\label{Holdcond2}
\Lambda(s)^{\frac{1}{N}} \Lambda_*^{-1}(\Lambda(s)) \leq Cs \quad \forall s>0.
\end{equation}
Recalling that $\Lambda_*^{-1}(\Lambda(s)) = \left(\int_0^s \left(\frac{\tau}{\Lambda(\tau)}\right)^{\frac{1}{N-1}} \dtau\right)^{\frac{1}{N'}}$, \eqref{Holdcond2} is equivalent to
\begin{equation}
\label{Holdcond3}
\Lambda(s)^{\frac{1}{N}} \left(\int_0^s \left(\frac{\tau}{\Lambda(\tau)}\right)^{\frac{1}{N-1}} \dtau\right)^{\frac{1}{N'}} \leq Cs \quad \forall s>0.
\end{equation}
By hypothesis, for any $\eps\in(0,N-s_\Lambda)$, the function $t\mapsto\frac{t^{N-\eps}}{\Lambda(t)}$ is increasing: indeed,
\begin{equation*}
\frac{{\rm d}}{\dt} \frac{t^{N-\eps}}{\Lambda(t)} = \frac{t^{N-\eps-1}}{\Lambda(t)} \left( N-\eps - \frac{t\Lambda'(t)}{\Lambda(t)} \right) \geq \frac{t^{N-\eps-1}}{\Lambda(t)} \left( N-s_\Lambda-\eps \right) > 0 \quad \forall t>0.
\end{equation*}
Thus, fixing $\eps$ as above,
\begin{equation*}
\begin{split}
\int_0^s \left(\frac{\tau}{\Lambda(\tau)}\right)^{\frac{1}{N-1}} \dtau &= \int_0^s \left(\frac{\tau^{N-\eps}}{\Lambda(\tau)}\right)^{\frac{1}{N-1}} \tau^{\frac{\eps}{N-1}-1} \dtau \leq \left(\frac{s^{N-\eps}}{\Lambda(s)}\right)^{\frac{1}{N-1}} \int_0^s \tau^{\frac{\eps}{N-1}-1} \dtau \\
&= \frac{N-1}{\eps} \frac{s^{N'}}{\Lambda(s)^{\frac{1}{N-1}}},
\end{split}
\end{equation*}
which ensures \eqref{Holdcond3}.
\end{proof}

\begin{prop}
\label{weightyoung}
Let $\Lambda$ be a Young function satisfying $s_\Lambda<N$, and let $w\in L^1(\R^N) \cap L^\infty(\R^N)$, $w\geq 0$. Then, for any $u\in \mathcal{D}^{1,\Lambda}_0(\R^N)$,
\begin{equation*}
\int_{\R^N} w\Lambda(|u|) \dx \leq C \int_{\R^N} \Lambda(|\nabla u|) \dx,
\end{equation*}
being $C=C(N,\|w\|_1,\|w\|_\infty)>0$. In particular, the embedding
\begin{equation*}
\mathcal{D}^{1,\Lambda}_0(\R^N) \hookrightarrow L^\Lambda(\R^N;w)
\end{equation*}
is continuous.
\end{prop}

\begin{proof}
Set $C_N=8|B_1|^{-1/N}$ and take any $\eps\in(0,1)$. Observe that
\begin{equation*}
\Lambda(|u|) = \left[\eps\frac{\Lambda(|u|)}{|u|}\right] \left[\frac{C_N}{\eps}\left(\int_{\R^N}\Lambda(|\nabla u|) \dx\right)^{\frac{1}{N}}\right] \left[\frac{|u|}{C_N \left(\int_{\R^N}\Lambda(|\nabla u|) \dx\right)^{\frac{1}{N}}}\right].
\end{equation*}
Now we apply Lemma \ref{youngineq} with the functions $\overline{\Lambda}$, $t^N$, and $\Lambda_*$, after observing that
\begin{equation*}
\overline{\Lambda}^{-1}(t)t^{\frac{1}{N}}\Lambda_*^{-1}(t) \leq k\overline{\Lambda}^{-1}(t)\Lambda^{-1}(t) \leq 2kt,
\end{equation*}
by virtue of \eqref{Holdcond} and \eqref{fundineq0}. Exploiting \eqref{fundineq2} We get
\begin{equation*}
\begin{aligned}
&\Lambda(|u|) \\
&\leq C\left[\eps\overline{\Lambda}\left(\frac{\Lambda(|u|)}{|u|}\right) + \left(\frac{C_N}{\eps}\right)^N \int_{\R^N}\Lambda(|\nabla u|) \dx + \Lambda_*\left(\frac{|u|}{C_N \left(\int_{\R^N}\Lambda(|\nabla u|) \dx\right)^{\frac{1}{N}}}\right)\right] \\
&\leq C\eps\Lambda(|u|) + C\left(\frac{C_N}{\eps}\right)^N \int_{\R^N}\Lambda(|\nabla u|) \dx + C\Lambda_*\left(\frac{|u|}{C_N \left(\int_{\R^N}\Lambda(|\nabla u|) \dx\right)^{\frac{1}{N}}}\right).
\end{aligned}
\end{equation*}
Re-absorbing the first term on the left, multiplying by $w$, and integrating yield
\begin{equation*}
\begin{aligned}
&(1-C\eps)\int_{\R^N} w\Lambda(|u|) \dx \\
&\leq C\left[\|w\|_1 \left(\frac{C_N}{\eps}\right)^N \int_{\R^N}\Lambda(|\nabla u|) \dx + \|w\|_\infty \int_{\R^N} \Lambda_*\left(\frac{|u|}{C_N \left(\int_{\R^N}\Lambda(|\nabla u|) \dx\right)^{\frac{1}{N}}}\right) \dx\right].
\end{aligned}
\end{equation*}
Now we choose $\eps=\frac{1}{2C}$. Using \cite[Theorem 3]{C} we conclude
\begin{equation*}
\int_{\R^N} w\Lambda(|u|) \dx \leq 2C\left[\left(\frac{C_N}{\eps}\right)^N \|w\|_1 + \|w\|_\infty\right] \int_{\R^N}\Lambda(|\nabla u|) \dx.
\end{equation*}
\end{proof}

\begin{prop}
\label{compemb}
Let $\Lambda$ be a Young function satisfying $s_\Lambda<N$, and let $w\in L^1(\R^N) \cap L^\infty(\R^N)$, $w\geq 0$, obeying $w(x) \to 0$ as $|x|\to+\infty$. Then the embedding
\begin{equation*}
\mathcal{D}^{1,\Lambda}_0(\R^N) \hookrightarrow L^\Lambda(\R^N;w)
\end{equation*}
is compact.
\end{prop}
\begin{proof}
This proof is inspired by \cite[Theorem 2.1]{V}. For any bounded domain $\Omega \subseteq \R^N$, we define the weighted Orlicz-Sobolev space
\begin{equation*}
W^{1,\Lambda}(\Omega;w) := \{u\in L^\Lambda(\Omega;w): \, |\nabla u|\in L^\Lambda(\Omega)\}.
\end{equation*}
The following operators are continuous:
\begin{equation*}
\begin{alignedat}{2}
&r_n: \mathcal{D}^{1,\Lambda}_0(\R^N) \to W^{1,\Lambda}(B_n;w), \quad &&r_n(u) = u_{\mid_{B_n}}, \\
&i_n: W^{1,\Lambda}(B_n;w) \to L^\Lambda(B_n;w), \quad &&i_n(u) = u, \\
&e_n: L^\Lambda(B_n;w) \to L^\Lambda(\R^N;w), \quad && e_n(u)(x) = \left\{
\begin{alignedat}{2}
&u(x) \quad &\mbox{if} \;\; x\in B_n, \\
&0 \quad &\mbox{if} \;\; x\in B_n^e. \\
\end{alignedat}
\right.
\end{alignedat}
\end{equation*}
Reasoning as in \cite[Lemma 7.4.1]{KJF}, we deduce that $i_n$ is compact; thus $I_n:=e_n\circ i_n\circ r_n$ is compact too. We want to prove that $\{I_n\}$ converges to the identity operator. To this aim, first we show that, for all $u\in L^{\Lambda_*}(\R^N)$,
\begin{equation}
\label{operatorineq}
\|u\|_{L^\Lambda(B_n^e;w)} \leq c_n \|u\|_{\Lambda_*}, \quad \mbox{being} \;\; c_n:=\|w\|_{L^\infty(B_n^e)}^{\frac{1}{2s_{\Lambda_*}}}.
\end{equation}
Observe that $\|w\|_{L^\infty(B_n^e)}\to 0$ because of $w(x)\to 0$ as $|x|\to+\infty$, as well as $\|w\|_{L^1(B_n^e)}\to 0$, since $w\in L^1(\R^N)$. Thus, choosing $M>0$ such that $\Lambda(t) \leq \Lambda_*(t)$ for all $t>M$, besides using \eqref{factor}, for any $n$ sufficiently large we have
\begin{equation*}
\begin{split}
&\int_{B_n^e} w\Lambda\left(\frac{|u|}{c_n\|u\|_{\Lambda_*}}\right) \dx \\
&\leq \int_{\{u>Mc_n\|u\|_{\Lambda_*}\} \cap B_n^e} w \Lambda_*\left(\frac{|u|}{c_n\|u\|_{\Lambda_*}}\right) \dx + \Lambda(M) \int_{\{u\leq Mc_n\|u\|_{\Lambda_*}\} \cap B_n^e} w \dx \\
&\leq C \left[\int_{B_n^e} \frac{w}{c_n^{s_{\Lambda_*}}} \Lambda_*\left(\frac{|u|}{\|u\|_{\Lambda_*}}\right) \dx + \int_{B_n^e} w \dx\right] \\
&\leq C \left[ \|w\|_{L^\infty(B_n^e)}^{\frac{1}{2}} + \|w\|_{L^1(B_n^e)} \right] \leq 1,
\end{split}
\end{equation*}
ensuring \eqref{operatorineq}. Accordingly, for any $n,m\in\N$ such that $m>n$,
\begin{equation*}
\|I_n-I_m\| = \sup_{u\in\mathcal{D}^{1,\Lambda}_0(\R^N)\setminus\{0\}} \frac{\|I_n(u)-I_m(u)\|_{L^\Lambda(\R^N;w)}}{\|u\|_{\mathcal{D}^{1,\Lambda}_0(\R^N)}} \leq C \sup_{u\in\mathcal{D}^{1,\Lambda}_0(\R^N)\setminus\{0\}} \frac{\|u\|_{L^\Lambda(B_n^e;w)}}{\|u\|_{L^{\Lambda_*}(\R^N)}} \leq C c_n,
\end{equation*}
where we have used the embedding inequality related to \eqref{embedding}. Since $c_n\to 0$, $\{I_n\}$ is a Cauchy sequence of compact operators, and hence $I_n\to I$ for some compact operator $I$ (see \cite[Theorem 6.1]{B}). The argument in \cite[p.283]{V} identifies $I$ with the identity operator, which concludes the proof.
\end{proof}

\section{Decay estimates}

\begin{lemma}
\label{radialsub}
Suppose \ref{ellipticity}. Then, for all $c,r>0$ there exists $\underline{w}\in C^1_{\rm loc}(\R^N \setminus\{0\})$ such that $\underline{w}_{\mid_{B_r^e}}\in\mathcal{D}^{1,\Phi}_0(B_r^e)$ and
\begin{equation}
\label{subprob}
\tag{${\rm\underline{P}}$}
\left\{
\begin{alignedat}{2}
-\Delta_{\Phi} \underline{w} &= 0 &&\quad \mbox{in}\;\; \R^N \setminus\{0\}, \\
\underline{w} &= c &&\quad \mbox{on}\;\; \partial B_r, \\
\underline{w}(x) &\to 0 &&\quad \mbox{as}\;\; |x| \to +\infty.
\end{alignedat}
\right.
\end{equation}
Moreover, $\underline{w}$ can be represented as
\begin{equation}
\label{radialsubrepr}
\underline{w}(x) = \int_{|x|}^{+\infty} \phi^{-1}(kr^{1-N}) \dr,
\end{equation}
being $k$ dependent only on $c>0$.
\end{lemma}

\begin{proof}
Let us look for radially decreasing solutions to \eqref{subprob}, i.e., $z:\R_+\to\R$ of class $C^1$ such that $\underline{w}(x) = z(s)$ and $z'(s)<0$, where $s=|x|$. Hence \eqref{subprob} can be rewritten as
\begin{equation*}
\left\{
\begin{alignedat}{2}
(s^{N-1}\phi(|z'(s)|))' &= 0 &&\quad \mbox{in}\;\; \R_+, \\
z(r) &= c, \\
z(s) &\to 0 &&\quad \mbox{as}\;\; s \to +\infty.
\end{alignedat}
\right.
\end{equation*}
Integrating we get, for a suitable $k>0$,
\begin{equation*}
s^{N-1}\phi(|z'(s)|)=k,
\end{equation*}
whence
\begin{equation*}
\phi(|z'(s)|)=ks^{1-N}.
\end{equation*}
Inverting $\phi$ and integrating again yield
\begin{equation*}
z(s) = \int_s^{+\infty} \phi^{-1}(kr^{1-N})\dr+k',
\end{equation*}
being $k'\in\R$ opportune. Observe that the change of variable $\tau=kr^{1-N}$, \eqref{invder}, \eqref{indices} for $\Lambda=\overline{\Phi}$, \eqref{factor}, and \eqref{youngind} yield
\begin{equation}
\label{asymptotics}
\begin{split}
\int_s^{+\infty} \phi^{-1}(kr^{1-N}) \dr &\leq C \int_0^{ks^{1-N}} \tau^{-N'} \phi^{-1}(\tau) \dtau = \int_0^{ks^{1-N}} \tau^{-N'} \overline{\Phi}'(\tau) \dtau \\
&\leq C \int_0^{ks^{1-N}} \tau^{-N'-1} \overline{\Phi}(\tau) \dtau \leq C_s \int_0^{ks^{1-N}} \tau^{-N'-1+i_{\overline{\Phi}}} \dtau \\
&\leq C_s \int_0^{ks^{1-N}} \tau^{-N'-1+s_{\Phi}'} \dtau <+\infty
\end{split}
\end{equation}
for all $s>0$, since $s_{\Phi}' > N'$ by \ref{ellipticity}. Accordingly,
\begin{equation}
\label{radialsumm}
r\mapsto\phi^{-1}(kr^{1-N}) \quad \mbox{belongs to} \;\; L^1(s,+\infty) \cap L^\infty(s,+\infty) \quad \mbox{for all} \;\; s>0.
\end{equation}
Thus, imposing that $\lim_{s\to\infty}z(s)=0$, we obtain
\begin{equation*}
z(s) = \int_s^{+\infty} \phi^{-1}(kr^{1-N})\dr.
\end{equation*}
Let us consider the function $\theta:\R_+\to\R_+$ defined as $\theta(k) := \int_r^{+\infty} \phi^{-1}(k\tau^{1-N})\dtau$, which is strictly increasing and continuous, according to Lebesgue's dominated convergence theorem. We observe that $\lim_{k\to0^+}\theta(k)=0$ and $\lim_{k\to+\infty}\theta(k)=+\infty$, by virtue of Beppo Levi's monotone convergence theorem. Hence there exists a unique $k>0$ such that $\theta(k)=c$, that is, $z(r)=c$; this concludes the proof of \eqref{radialsubrepr}.

Since $z'(s)=-\phi^{-1}(ks^{1-N})$ is continuous, then $\underline{w}\in C^1_{\rm loc}(\R^N \setminus \{0\})$. By \eqref{radialsubrepr}, \eqref{radialsumm}, and Proposition \ref{interpolation}, it follows that
\begin{equation}
\label{gradsumm}
\left|\nabla \underline{w}_{\mid_{B_r^e}}\right|\in L^\Phi(B_r^e).
\end{equation}
In order to show that $\underline{w}\in\mathcal{D}^{1,\Phi}_0(B_r^e)$, we consider the sequence $\{w_n\} \subseteq C^\infty_c(\R^N)$ defined via convolution by
\begin{equation*}
w_n := -\left(\int_{|\cdot|}^{+\infty} z'(r) \chi_{[0,n]}(r) \dr\right) *\rho_n \quad \forall n\in\N,
\end{equation*}
and prove that $\|\nabla w_n-\nabla \underline{w}\|_{L^\Phi(B_r^e)} \to 0$ as $n\to\infty$. Indeed, Young's convolution inequality (see the proof of \cite[Corollary VI.3.7]{RR}) and the fact that $\|\rho_n\|_1 = 1$ for all $n\in\N$ imply
\begin{equation*}
\begin{split}
\|\nabla w_n-\nabla \underline{w}\|_{L^\Phi(B_r^e)} &= \|(\chi_{B_n}\nabla\underline{w})*\rho_n-\nabla\underline{w}\|_{L^\Phi(B_r^e)} \\
&= \|(\chi_{B_n}\nabla\underline{w})*\rho_n-(\nabla\underline{w})*\rho_n+(\nabla\underline{w})*\rho_n-\nabla\underline{w}\|_{L^\Phi(B_r^e)} \\
&\leq \|(\chi_{B_n^e}\nabla\underline{w})*\rho_n\|_{L^\Phi(B_r^e)} + \|(\nabla\underline{w})*\rho_n-\nabla\underline{w}\|_{L^\Phi(B_r^e)} \\
&\leq \|\chi_{B_n^e}\nabla\underline{w}\|_{L^\Phi(B_r^e)} + \|(\nabla\underline{w})*\rho_n-\nabla\underline{w}\|_{L^\Phi(B_r^e)}.
\end{split}
\end{equation*}
Using \eqref{gradsumm} and the properties of mollifiers (see, e.g., \cite[Theorem 3.18.1.1]{KJF}), we infer $\chi_{B_n^e}\nabla\underline{w}\to0$ and $(\nabla\underline{w})*\rho_n\to\nabla\underline{w}$ in $L^\Phi(B_r^e)$ as $n\to\infty$; thus we get $\nabla w_n\to \nabla\underline{w}$ in $L^\Phi(B_r^e)$.
\end{proof}

\begin{lemma}
\label{radialsuper}
Let \ref{ellipticity} be satisfied. Then, for all $c,\mu,r>0$ and $l>N$, there exists $\overline{w}\in C^1_{\rm loc}(\R^N \setminus\{0\})$ such that $\overline{w}_{\mid_{B_r^e}}\in\mathcal{D}^{1,\Phi}_0(B_r^e)$ and
\begin{equation}
\label{superprob}
\tag{${\rm\overline{P}}$}
\left\{
\begin{alignedat}{2}
-\Delta_{\Phi} \overline{w} &= \mu|x|^{-l} &&\quad \mbox{in}\;\; \R^N \setminus\{0\}, \\
\overline{w} &= c &&\quad \mbox{on}\;\; \partial B_r, \\
\overline{w}(x) &\to 0 &&\quad \mbox{as}\;\; |x| \to +\infty,
\end{alignedat}
\right.
\end{equation}
provided
\begin{equation}
\label{boundarycond}
c>\int_r^{+\infty} \phi^{-1}\left(\frac{\mu}{l-N}\tau^{1-l}\right) \dtau.
\end{equation}
Moreover, $\overline{w}$ can be represented as
\begin{equation}
\label{radialsuperrepr}
\overline{w}(x) = \int_{|x|}^{+\infty} \phi^{-1}\left(\frac{\mu}{1-l}r^{1-l}+kr^{1-N}\right)\dr,
\end{equation}
being $k$ dependent only on $c,\mu>0$.
\end{lemma}

\begin{proof}
We reason as in the proof of Lemma \ref{radialsub}, looking for $z:\R_+\to\R$ of class $C^1$ such that $\underline{w}(x) = z(s)$ and $z'(s)<0$, $s=|x|$, satisfying
\begin{equation*}
\left\{
\begin{alignedat}{2}
-(s^{N-1}\phi(|z'(s)|))' &= \mu s^{N-1-l} &&\quad \mbox{in}\;\; \R_+, \\
z(r) &= c, \\
z(s) &\to 0 &&\quad \mbox{as}\;\; s \to +\infty.
\end{alignedat}
\right.
\end{equation*}
Integrating we get, for a suitable $k>0$,
\begin{equation*}
s^{N-1}\phi(|z'(s)|)=\frac{\mu}{l-N}s^{N-l}+k,
\end{equation*}
whence
\begin{equation*}
\phi(|z'(s)|)=\frac{\mu}{l-N}s^{1-l}+ks^{1-N}.
\end{equation*}
Inverting $\phi$ and integrating again produce
\begin{equation*}
z(s) = \int_s^{+\infty} \phi^{-1}\left(\frac{\mu}{l-N}r^{1-l}+kr^{1-N}\right)\dr+k',
\end{equation*}
being $k'\in\R$ opportune. Observe that, for any $r>s$,
\begin{equation*}
\phi^{-1}\left(\frac{\mu}{l-N}r^{1-l}+kr^{1-N}\right) \leq C_s\phi^{-1}(kr^{1-N}).
\end{equation*}
Thus, recalling \eqref{radialsumm} and imposing that $\lim_{s\to\infty}z(s)=0$, we get
\begin{equation*}
z(s) = \int_s^{+\infty} \phi^{-1}\left(\frac{\mu}{l-N}r^{1-l}+kr^{1-N}\right)\dr.
\end{equation*}
Let us consider the function $\theta:\R_+\to\R_+$ defined as $\theta(k):=\int_r^{+\infty} \phi^{-1}\left(\frac{\mu}{l-N}\tau^{1-l}+ks^{1-N}\right)\dtau$. Observe that $\lim_{k\to0^+} \theta(k) =\int_r^{+\infty} \phi^{-1}\left(\frac{\mu}{l-N}\tau^{1-l}\right)\dtau<c$, according to Beppo Levi's theorem and \eqref{boundarycond}. Then, repeating verbatim the arguments in the proof of Lemma \ref{radialsub}, we find a unique $k>0$ such that $\overline{w}=c$ on $\partial B_r$, ensuring \eqref{radialsuperrepr}, and guarantee that $\overline{w} \in C^1_{\rm loc}(\R^N \setminus \{0\})$, as well as $\overline{w}_{\mid_{B_r^e}}\in\mathcal{D}^{1,\Phi}_0(B_r^e)$.
\end{proof}

\begin{rmk}
\label{decayest}
By \eqref{radialsubrepr} and \eqref{radialsuperrepr} we deduce some decay estimates. Indeed, fixing any $r>0$, by \ref{ellipticity} it turns out that
\begin{equation*}
C_r^{-1}\int_{|x|}^{+\infty} \phi^{-1}(\tau^{1-N})\dtau \leq \underline{w}(x),\overline{w}(x) \leq C_r\int_{|x|}^{+\infty} \phi^{-1}(\tau^{1-N})\dtau \quad \mbox{in} \;\; B_r^e.
\end{equation*}
Reasoning as in \eqref{asymptotics}, we infer
\begin{equation*}
\int_{|x|}^{+\infty} \phi^{-1}(\tau^{1-N})\dtau \leq C_r \int_0^{|x|^{1-N}} \tau^{s_\Phi'-N'-1} \dtau \leq C_r |x|^{\frac{s_\Phi-N}{s_\Phi-1}}
\end{equation*}
and
\begin{equation*}
\int_{|x|}^{+\infty} \phi^{-1}(\tau^{1-N})\dtau \geq C_r^{-1} \int_0^{|x|^{1-N}} \tau^{i_\Phi'-N'-1} \dtau \geq C_r^{-1} |x|^{\frac{i_\Phi-N}{i_\Phi-1}}.
\end{equation*}
Hence,
\begin{equation*}
C_r^{-1}|x|^{\frac{i_\Phi-N}{i_\Phi-1}} \leq \underline{w}(x),\overline{w}(x) \leq C_r|x|^{\frac{s_\Phi-N}{s_\Phi-1}} \quad \mbox{in} \;\; B_r^e.
\end{equation*}
We highlight the fact that these estimates are coherent with the ones in \cite[Theorems 3 and C]{AB} for the $p$-Laplacian, that is, $i_\Phi=s_\Phi=p$.
\end{rmk}

\section{Existence result}

\subsection{The regularized problem}
For any $\eps\in(0,1)$, let us consider the following system:
\begin{equation}
\label{regularprob}
\tag{${\rm P_\eps}$}
\left\{
\begin{alignedat}{2}
-\Delta_{\Phi} u &= h(x)f(u_++\eps,v_++\eps) &&\quad \mbox{in}\;\; \R^N, \\
-\Delta_{\Psi} v &= k(x)g(u_++\eps,v_++\eps) &&\quad \mbox{in}\;\; \R^N, \\
u(x) &\to 0 &&\quad \mbox{as}\;\; |x| \to +\infty, \\
v(x) &\to 0 &&\quad \mbox{as}\;\; |x| \to +\infty.
\end{alignedat}
\right.
\end{equation}

The energy functional associated with \eqref{regularprob} is $J:X \to \R$ defined as
\begin{equation*}
J_\eps(u,v) = \int_{\R^N} \Phi(|\nabla u|) \dx + \int_{\R^N} \Psi(|\nabla v|) \dx - \int_{\R^N} H(x,u_++\eps,v_++\eps) \dx.
\end{equation*}

\begin{lemma}
\label{potentialest}
Under \ref{ellipticity} and \ref{varstruct}--\ref{growthcond}, for all $\sigma>0$ the following estimate holds true:
\begin{equation*}
|H(x,s,t)| \leq l(x) \left[\sigma\left(\Phi(s) + \Psi(t)\right) + C_\sigma\right]
\end{equation*}
for all $(x,s,t)\in\R^N\times(0,+\infty)^2$, where $l:=h+k$ and $C_\sigma>0$ is a suitable constant.
\end{lemma}

\begin{proof}
Set $l:=h+k$. Using Torricelli's theorem, \ref{varstruct}, the convexity of $\Upsilon_i$, $\Gamma_i$, $i=1,2$, and \eqref{strictdom} with $\Lambda_1(t)=t$ and $\Lambda_2=\Upsilon_1$ (resp., $\Lambda_2=\Gamma_2$), we get
\begin{equation*}
\begin{split}
|H(x,s,t)| &\leq \int_0^1 \left|s\partial_s H(x,rs,rt) + t\partial_t H(x,rs,rt)\right| \dr \\
&\leq \int_0^1 \left[ h(x)sf(rs,rt)+k(x)tg(rs,rt) \right] \dr \\
&\leq h(x)\int_0^1\left((r^{-\alpha}s^{1-\alpha}+s)\left(\frac{\Upsilon_2(rt)}{rt}+1\right)+\frac{\Upsilon_1(rs)}{r} \right) \dr \\
&\quad + k(x) \int_0^1\left((r^{-\beta}t^{1-\beta}+t)\left(\frac{\Gamma_1(rs)}{rs}+1\right)+\frac{\Gamma_2(rt)}{r} \right) \dr \\
&\leq h(x)\int_0^1\left((r^{-\alpha}s^{1-\alpha}+s)\left(\frac{\Upsilon_2(t)}{t}+1\right)+\Upsilon_1(s) \right) \dr \\
&\quad + k(x) \int_0^1\left((r^{-\beta}t^{1-\beta}+t)\left(\frac{\Gamma_1(s)}{s}+1\right)+\Gamma_2(t) \right) \dr \\
&\leq Cl(x)\left[(s^{1-\alpha}+s)\left(\frac{\Upsilon_2(t)}{t}+1\right)+\Upsilon_1(s) + (t^{1-\beta}+t)\left(\frac{\Gamma_1(s)}{s}+1\right)+\Gamma_2(t) \right] \\
&\leq Cl(x)\left[(s+1)\frac{\Upsilon_2(t)}{t}+(t+1)\frac{\Gamma_1(s)}{s}+\Upsilon_1(s)+\Gamma_2(t)+1\right].
\end{split}
\end{equation*}
According to Young's inequality and \eqref{subadd} (with $\Lambda=\overline{\Lambda}_i$, $s=w$, and $t=1$), besides \eqref{fundineq}, we deduce
\begin{equation*}
\begin{split}
&|H(x,s,t)| \\
&\leq Cl(x)\left[ \overline{\Lambda}_1(s)+\Lambda_1\left(\frac{\Upsilon_2(t)}{t}\right)+\overline{\Lambda}_2(t)+\Lambda_2\left(\frac{\Gamma_1(s)}{s}\right)+\Upsilon_1(s)+\Gamma_2(t)+1 \right] \\
&\leq Cl(x)\left[\overline{\Lambda}_1(s)+(\Lambda_1\circ \overline{\Upsilon}_2^{-1}\circ \Upsilon_2)(t)+\overline{\Lambda}_2(t)+(\Lambda_2\circ \overline{\Gamma}_1^{-1}\circ \Gamma_1)(s)+\Upsilon_1(s)+\Gamma_2(t)+1\right]. 
\end{split}
\end{equation*}
Let $\sigma>0$. Proposition \ref{lambdadef} ensures that $\overline{\Lambda}_1 \ll \Phi$ and $\overline{\Lambda}_2 \ll \Psi$, as well as $\Lambda_1\circ \overline{\Upsilon}_2^{-1}\circ \Upsilon_2 \ll \Psi$ and $\Lambda_2\circ \overline{\Gamma}_1^{-1}\circ \Gamma_1\ll \Phi$. Hence, recalling also $\Upsilon_1\ll\Phi$ and $\Gamma_2\ll\Psi$ by \ref{growthcond}, a repeated application of \eqref{strictdom} yields
\begin{equation*}
|H(x,s,t)| \leq l(x)\left[\sigma(\Phi(s)+\Psi(t))+C_\sigma\right].
\end{equation*}
\end{proof}

\begin{lemma}
\label{functprops}
Let \hyperlink{H1}{${\rm (H_1)}$}, \ref{varstruct}--\ref{growthcond}, and \ref{weightscond} be satisfied. Then $J_\eps$ is well defined, weakly sequentially lower semi-continuous, coercive, and of class $C^1$.
\end{lemma}

\begin{proof}
In order to prove that $J_\eps$ is well defined, it suffices to prove that $\int_{\R^N} |H(x,u_++\eps,v_++\eps)| \dx$ is finite for all $(u,v) \in X$. By Lemma \ref{potentialest} with $\sigma=1$, \eqref{subadd}, \ref{weightscond}, and Proposition \ref{weightyoung}, we get
\begin{equation}
\label{wellposed}
\begin{split}
&\int_{\R^N} |H(x,u_++\eps,v_++\eps)| \dx \leq \int_{\R^N} l\left[\Phi(u_++\eps)+\Psi(v_++\eps)+C\right]\dx \\
&\leq C \int_{\R^N} l\left[\Phi(u_+)+\Psi(v_+)+1\right]\dx \leq C \left(\int_{\R^N} l\Phi(|u|)\dx + \int_{\R^N}l\Psi(|v|)\dx+ \|l\|_1\right) \\
&\leq C \left(\int_{\R^N} \Phi(|\nabla u|) \dx + \int_{\R^N} \Psi(|\nabla v|) \dx + 1\right) < +\infty.
\end{split}
\end{equation}

Now we prove that $J_\eps$ is weakly sequentially lower semi-continuous. Take any $\{(u_n,v_n)\}\subseteq X$ and $(u,v)\in X$ such that $(u_n,v_n)\rightharpoonup(u,v)$ in $X$. A diagonal argument (see, e.g., \cite[Lemma 3.5]{GMM}) ensures that $(u_n,v_n)\to(u,v)$ a.e. in $\R^N$. By \eqref{wellposed} we have
\begin{equation*}
\int_{\R^N} |H(x,(u_n)_++\eps,(v_n)_++\eps)| \dx \leq C \left( \int_{\R^N} l\left[\Phi(|u_n|) + \Psi(|v_n|) \right]\dx + 1 \right).
\end{equation*}
According to Proposition \ref{compemb} and \cite[Proposition 26.2(b)]{Z/IIB}, we deduce $u_n \to u$ and $v_n \to v$ in $L^\Phi(\R^N;l)$ and $L^\Psi(\R^N;l)$, respectively. Thus, by \cite[Corollary 3.3.4]{HH},
\begin{equation*}
\int_{\R^N} l\Phi(|u_n|) \dx \to \int_{\R^N} l\Phi(|u|) \dx \quad \mbox{and} \quad \int_{\R^N} l\Psi(|v_n|) \dx \to \int_{\R^N} l\Psi(|v|) \dx.
\end{equation*}
Then Pratt's lemma (see Proposition \ref{pratt}) entails
\begin{equation*}
\int_{\R^N} H(x,(u_n)_++\eps,(v_n)_++\eps) \dx \to \int_{\R^N} H(x,u_++\eps,v_++\eps) \dx.
\end{equation*}
For any Young function $\Theta$, the functional $w \mapsto \int_{\R^N} \Theta(|\nabla w|) \dx$ is weakly sequentially lower semi-continuous in $\mathcal{D}^{1,\Theta}_0(\R^N)$; cf. \cite[Lemma 2.6]{CGL}. Then, exploiting Fatou's Lemma, we deduce
\begin{equation*}
\begin{split}
\liminf_{n\to\infty} J_\eps(u_n,v_n) &\geq \liminf_{n\to\infty} \int_{\R^N} \Phi(|\nabla u_n|) \dx + \liminf_{n\to\infty} \int_{\R^N} \Psi(|\nabla v_n|) \dx \\
&\quad - \limsup_{n\to\infty} \int_{\R^N} H(x,(u_n)_++\eps,(v_n)_++\eps) \dx \\
&\geq \int_{\R^N} \Phi(|\nabla u|) \dx + \int_{\R^N} \Psi(|\nabla v|) \dx - \int_{\R^N} H(x,u_++\eps,v_++\eps) \dx \\
&= J_\eps(u,v),
\end{split}
\end{equation*}
which proves the weak sequential lower semi-continuity of $J_\eps$.

Let us prove coercivity of $J_\eps$. Reasoning as in \eqref{wellposed} and exploiting Proposition \ref{weightyoung} we have
\begin{equation*}
\begin{split}
&J_\eps(u,v) \\
&\geq \int_{\R^N} \Phi(|\nabla u|) \dx + \int_{\R^N} \Psi(|\nabla v|) \dx - \int_{\R^N} |H(x,u_++\eps,v_++\eps)| \dx \\
&\geq \int_{\R^N} \Phi(|\nabla u|) \dx + \int_{\R^N} \Psi(|\nabla v|) \dx - C\sigma\left(\int_{\R^N} l\Phi(|u|) \dx + \int_{\R^N} l\Psi(|v|) \dx\right) - C_\sigma \\
&\geq (1-C\sigma) \left[ \int_{\R^N} \Phi(|\nabla u|) \dx + \int_{\R^N} \Psi(|\nabla v|) \dx \right] - C_\sigma.
\end{split}
\end{equation*}
Thus, choosing $\sigma$ small enough yields coercivity of $J_\eps$.

To conclude, we prove that $J_\eps$ is of class $C^1$. According to \cite[Lemma 2.6]{CGL}, the functionals $u \mapsto \int_{\R^N} \Phi(|\nabla u|) \dx$ and $v \mapsto \int_{\R^N} \Psi(|\nabla v|) \dx$ enjoy this property. Concerning the functional $\mathcal{H}_\eps:X\to\R$ defined as
\begin{equation*}
\mathcal{H}_\eps(u,v) := \int_{\R^N} H(x,u_++\eps,v_++\eps) \dx,
\end{equation*}
we compute its G\^{a}teaux derivative as follows. Given $(w,z)\in X$, by Torricelli's theorem we have
\begin{equation}
\label{torricelli}
\begin{split}
\lim_{t\to 0^+} &\frac{1}{t}\left[\int_{\R^N} H(x,(u+tw)_++\eps,(v+tz)_++\eps) \dx - \int_{\R^N} H(x,u_++\eps,v_++\eps) \dx\right] \\
&= \lim_{t\to 0^+} \int_{\R^N} \frac{1}{t}\left[H(x,(u+tw)_++\eps,(v+tz)_++\eps)-H(x,u_++\eps,v_++\eps)\right] \dx \\
&= \lim_{t\to 0^+} \int_{\R^N} \left( \int_0^1 \left[h(x)f((u+stw)_++\eps,(v+stz)_++\eps)w \right.\right. \\
&\quad \left.+ k(x)g((u+stw)_++\eps,(v+stz)_++\eps)z\right] \ds \bigg) \dx.
\end{split}
\end{equation}
Without loss of generality, assume that $t\in(0,1)$. Exploiting \ref{growthcond}, Young's inequality, and \eqref{subadd}, we get
\begin{equation}
\label{summability}
\begin{split}
&h(x)f((u+stw)_++\eps,(v+stz)_++\eps)|w| \\
&\leq Ch(x)\left[ (1+\eps^{-\alpha})\left(\frac{\Upsilon_2((v+stz)_++\eps)}{(v+stz)_++\eps}+1\right) + \frac{\Upsilon_1((u+stw)_++\eps)}{(u+stw)_++\eps} \right]|w|\\
&\leq C_\eps h(x)\left[\frac{\Upsilon_2((v+stz)_++\eps)}{(v+stz)_++\eps} + \frac{\Upsilon_1((u+stw)_++\eps)}{(u+stw)_++\eps} + 1 \right]|w|\\
&\leq C_\eps h(x)\Big[ (\overline{\Phi}\circ \overline{\Upsilon}_2^{-1}\circ \Upsilon_2)((v+stz)_++\eps) + \Phi(|w|) \Big. \\
&\quad\Big. +\Upsilon_1((u+stw)_++\eps) + \Upsilon_1(|w|) + |w| \Big]\\
&\leq C_\eps h(x)\left[ \Psi((v+stz)_++\eps) + \Phi((u+stw)_++\eps) + 2\Phi(|w|) + |w| \right]\\
&\leq C_\eps h(x)\left[\Psi(|v|) + \Psi(|z|) + \Phi(|u|) + \Phi(|w|) + 1 \right], \\
\end{split}
\end{equation}
which is summable according to Proposition \ref{weightyoung} (as for \eqref{wellposed}). A similar computation holds also for the term $k(x)g((u+stw)_++\eps,(v+stz)_++\eps)|z|$. Hence we are in the position to apply Lebesgue's theorem to \eqref{torricelli}, since the right-hand side of \eqref{summability} belongs to $L^1(\R^N)$ and it is independent of both $s$ and $t$. We deduce
\begin{equation*}
\begin{split}
&\langle \mathcal{H}_\eps'(u,v),(w,z) \rangle\\
&=\lim_{t\to 0^+} \frac{1}{t}\left[\int_{\R^N} H(x,(u+tw)_++\eps,(v+tz)_++\eps) \dx - \int_{\R^N} H(x,u_++\eps,v_++\eps) \dx\right] \\
&= \int_{\R^N} \int_0^1 \lim_{t\to 0^+} \left[ h(x)f((u+stw)_++\eps,(v+stz)_++\eps)w \right.\\
&\quad \left. + k(x)g((u+stw)_++\eps,(v+stz)_++\eps)z \right] \dx \ds \\
&= \int_{\R^N} \left[ h(x)f(u_++\eps,v_++\eps)w + k(x)g(u_++\eps,v_++\eps)z \right] \dx.
\end{split}
\end{equation*}

It remains to prove that $\mathcal{H}_\eps':X \to X^*$ is continuous. Pick any $\{(u_n,v_n)\}\subseteq X$ and $(u,v)\in X$ such that $(u_n,v_n) \to (u,v)$ in $X$. In particular, $(u_n,v_n) \to (u,v)$ in $Y$ and almost everywhere in $\R^N$. According to \cite[Lemma 7.3]{SGC}, there exists $(U,V)\in Y$ such that
\begin{equation*}
|u_n| \leq U \quad \mbox{and} \quad |v_n| \leq V \quad \mbox{in } \;\; \R^N.
\end{equation*}
Take any $(w,z)\in X$ such that $\|(w,z)\|_X=1$. Reasoning as in \eqref{summability}, through \eqref{subadd} and $\overline{\Phi}<\overline{\Upsilon}_1$, we infer
\begin{equation}
\label{summability2}
\begin{split}
&h\overline{\Phi}\left(|f((u_n)_++\eps,(v_n)_++\eps)-f(u_++\eps,v_++\eps)|\right) \\
&\leq h \left[\overline{\Phi}\left(|f((u_n)_++\eps,(v_n)_++\eps)|\right) + \overline{\Phi}\left(|f(u_++\eps,v_++\eps)|\right) \right] \\
&\leq C_\eps h[\Psi(|v_n|) + \Phi(|u_n|) + \Psi(|v|) + \Phi(|u|) + 1] \\
&\leq C_\eps h[\Psi(|V|) + \Phi(|U|) + 1],
\end{split}
\end{equation}
which is summable, according to Proposition \ref{weightyoung} and \eqref{embedding}, and independent of $n$. The same argument applies to $g((u_n)_++\eps,(v_n)_++\eps)-g(u_++\eps,v_++\eps)$. Hence, applying H\"older's inequality \cite[Lemma 3.2.11]{HH} in the weighted spaces $L^\Phi(\R^N;h)$ and $L^{\overline{\Phi}}(\R^N;h)$ (resp., $L^\Psi(\R^N;k)$ and $L^{\overline{\Psi}}(\R^N;k)$), besides using Proposition \ref{weightyoung}, we obtain
\begin{equation}
\label{holder}
\begin{split}
&|\mathcal{H}_\eps'(u_n,v_n),(w,z) \rangle-\langle \mathcal{H}_\eps'(u,v),(w,z) \rangle| \\
&\leq \|f((u_n)_++\eps,(v_n)_++\eps)-f(u_++\eps,v_++\eps)\|_{L^{\overline{\Phi}}(\R^N;h)} \|w\|_{L^\Phi(\R^N;h)} \\
&\quad + \|g((u_n)_++\eps,(v_n)_++\eps)-g(u_++\eps,v_++\eps)\|_{L^{\overline{\Psi}}(\R^N;k)} \|z\|_{L^\Psi(\R^N;k)} \\
&\leq C\|f((u_n)_++\eps,(v_n)_++\eps)-f(u_++\eps,v_++\eps)\|_{L^{\overline{\Phi}}(\R^N;h)} \\
&\quad + C\|g((u_n)_++\eps,(v_n)_++\eps)-g(u_++\eps,v_++\eps)\|_{L^{\overline{\Psi}}(\R^N;k)}.
\end{split}
\end{equation}
According to \eqref{summability2}--\eqref{holder} and Lebesgue's theorem, we deduce
\begin{equation*}
\|\mathcal{H}_\eps'(u_n,v_n)-\mathcal{H}_\eps'(u,v)\|_{X^*} = \sup_{\|(w,z)\|_X=1} |\mathcal{H}_\eps'(u_n,v_n),(w,z) \rangle-\langle \mathcal{H}_\eps'(u,v),(w,z) \rangle| \to 0
\end{equation*}
as $n\to\infty$, proving the continuity of $\mathcal{H}'$. Summarizing, $J_\eps$ is of class $C^1$.
\end{proof}

\begin{thm}
\label{regexistence}
Under \hyperlink{H1}{${\rm (H_1)}$}, \ref{varstruct}--\ref{growthcond}, and \ref{weightscond}, for all $\eps\in(0,1)$ there exists $(u_\eps,v_\eps)\in X$ solution to \eqref{regularprob}. Moreover, any solution $(u_\eps,v_\eps)$ to \eqref{regularprob} is non-negative (i.e., $u_\eps,v_\eps\geq 0$ a.e. in $\R^N$).
\end{thm}

\begin{proof}
It suffices to apply the Weierstrass-Tonelli theorem \cite[Theorem 1.2]{S} to $J_\eps$, besides recalling Lemma \ref{functprops}. The fact that any $(u_\eps,v_\eps)$ is non-negative is a consequence of the non-negativity of the right-hand side of \eqref{prob}, according to the weak maximum principle (see, e.g., \cite[Theorem 3.2.2]{PS}).
\end{proof}

\subsection{A priori estimates}

\begin{lemma}
\label{energyest}
Suppose \ref{ellipticity}, \ref{varstruct}--\ref{growthcond}, and \ref{weightscond}. Then there exists $L>0$ independent of $\eps\in(0,1)$ such that
\begin{equation*}
\|(u_\eps,v_\eps)\|_X \leq L
\end{equation*}
for any $(u_\eps,v_\eps)$ solution to \eqref{regularprob}.
\end{lemma}

\begin{proof}
Fix $\eps\in(0,1)$ and set $(u,v):=(u_\eps,v_\eps)$. According to Theorem \ref{regexistence} we have $u,v \geq 0 $ a.e. in $\R^N$. Fix any $\sigma>0$. Testing the first equation of \eqref{regularprob} with $u$, reasoning as in the proof of Lemma \ref{potentialest}, and using Proposition \ref{weightyoung} yield
\begin{equation}
\label{enest1}
\begin{split}
\int_{\R^N} \Phi(|\nabla u|) \dx &\leq C\int_{\R^N} \phi(|\nabla u|)|\nabla u| \dx = C\int_{\R^N} h(x)f(u+\eps,v+\eps)u \dx \\
&\leq C\int_{\R^N} l(x) \left[ \sigma(\Phi(u+\eps)+\Psi(v+\eps))+C_\sigma \right] \dx \\
&\leq C\int_{\R^N} l(x) \left[ \sigma(\Phi(u)+\Psi(v))+C_\sigma \right] \dx \\
&\leq C\sigma \left[\int_{\R^N}\Phi(|\nabla u|) \dx + \int_{\R^N} \Psi(|\nabla v|) \dx \right] + C_\sigma.
\end{split}	
\end{equation}
Analogously, for the second equation we get
\begin{equation}
\label{enest2}
\int_{\R^N} \Psi(|\nabla v|) \dx \leq C\sigma \left[\int_{\R^N}\Phi(|\nabla u|) \dx + \int_{\R^N} \Psi(|\nabla v|) \dx \right] + C_\sigma.
\end{equation}
Summing \eqref{enest1}--\eqref{enest2} and re-arranging the terms gives
\begin{equation*}
(1-C\sigma) \left[ \int_{\R^N} \Phi(|\nabla u|) \dx + \int_{\R^N} \Psi(|\nabla v|) \dx \right] \leq C_\sigma.
\end{equation*}
Choosing $\sigma$ small enough permits to conclude.
\end{proof}

\begin{lemma}
\label{talenti}
Let $\Omega$ be an open subset of $\R^N$ satisfying $|\Omega|<+\infty$, $\Phi$ be a Young function of class $C^1$ such that $1<i_\Phi\leq s_\Phi<N$, and $f\in L^\delta(\Omega)$, $f\geq 0$, with $\delta>\frac{N}{i_\Phi}$. Then the unique solution $u\in \mathcal{D}^{1,\Phi}_0(\Omega)$ of
\begin{equation*}
\left\{
\begin{alignedat}{2}
-\Delta_\Phi u &= f(x) \quad &&\mbox{in} \;\; \Omega, \\
u &= 0 \quad &&\mbox{on} \;\; \partial \Omega,
\end{alignedat}
\right.
\end{equation*}
belongs to $L^\infty(\Omega)$. Moreover, $u$ satisfies the estimate
\begin{equation*}
\|u\|_\infty \leq C\left(\|f\|_{L^\delta(\Omega)}^{\frac{1}{i_\Phi-1}}+1\right),
\end{equation*}
for a suitable $C=C(|\Omega|,N,\Phi)>0$.
\end{lemma}

\begin{proof}
Let us consider the symmetrized problem
\begin{equation*}
\left\{
\begin{alignedat}{2}
-\Div\left(\frac{\Phi(|\nabla w|)}{|\nabla w|^2} \nabla w\right) &= f^*(r) \quad &&\mbox{in}\;\; B_\rho, \\
w &= 0 \quad &&\mbox{on} \;\; \partial B_\rho,
\end{alignedat}
\right.
\end{equation*}
being $\rho>0$ such that $|B_\rho| = |\Omega|$. Set $\Lambda(t):=\frac{\Phi(t)}{t}$. According to \eqref{fundineq}, we have $\Lambda(t) \geq \frac{1}{2}\overline{\Phi}^{-1}(\Phi(t))$ for all $t>0$, so
\begin{equation}
\label{inversefunct}
\Lambda^{-1}(t) \leq \Phi^{-1}(\overline{\Phi}(2t)) \leq C \Phi^{-1}(\overline{\Phi}(t)) \quad \forall t>0.
\end{equation}
Hence, in order to apply \cite[Theorem 3.1]{C2}, we only need to prove that
\begin{equation*}
\int_0^{|\Omega|} \overline{\Phi}\left(t^{\frac{1}{N}} f^{**}(t)\right) \dt < +\infty.
\end{equation*}
Since $|\Omega|<+\infty$, we can suppose that $\delta<N$ without loss of generality. Observe that
\begin{equation}
\label{powercompute}
\left(\frac{1}{N}-\frac{1}{\delta}\right) i_\Phi' > \frac{1-i_\Phi}{N} \frac{i_\Phi}{i_\Phi-1} = - \frac{i_\Phi}{N} \geq -\frac{s_\Phi}{N} > -1.
\end{equation}
Then \eqref{birearrang}, \eqref{factor}, \eqref{youngind}, and \eqref{powercompute} yield
\begin{equation}
\label{verif1}
\begin{split}
\int_0^{|\Omega|} \overline{\Phi}\left(t^{\frac{1}{N}} f^{**}(t)\right) \dt &\leq \int_0^{|\Omega|} \overline{\Phi}\left(t^{\frac{1}{N}-\frac{1}{\delta}} \|f\|_{L^\delta(\Omega)} \right) \dt \\
&\leq \int_0^{|\Omega|} \overline{\Phi}\left(t^{\frac{1}{N}-\frac{1}{\delta}} \|f\|_{L^\delta(\Omega)} +1 \right) \dt \\
&\leq C \int_0^{|\Omega|} \left(t^{\frac{1}{N}-\frac{1}{\delta}} \|f\|_{L^\delta(\Omega)} +1 \right)^{s_{\overline{\Phi}}} \dt \\
&\leq C \left(\|f\|_{L^\delta(\Omega)}^{i_\Phi'}\int_0^{|\Omega|} t^{\left(\frac{1}{N}-\frac{1}{\delta}\right)i_\Phi'} \dt + 1\right) < +\infty.
\end{split}
\end{equation}
Hence \cite[Theorem 3.1 and (2.21)]{C2} and \eqref{inversefunct} entail
\begin{equation}
\label{talentiest}
u^*(0) \leq w(0) \leq C \int_{0}^{|\Omega|} t^{-\frac{1}{N'}} \Phi^{-1}\left(\overline{\Phi}\left(t^{\frac{1}{N}}f^{**}(t)\right)\right) \dt.
\end{equation}
Miming the proof of \eqref{nonyoungcompind} and using \eqref{youngind} we deduce
\begin{equation*}
\frac{s\left(\Phi^{-1}\circ\overline{\Phi}\right)'(s)}{\left(\Phi^{-1}\circ\overline{\Phi}\right)(s)} \leq \frac{s_{\overline{\Phi}}}{i_\Phi} \leq \frac{i_\Phi'}{i_\Phi} = \frac{1}{i_\Phi-1} \quad \forall t>0.
\end{equation*}
Dividing by $s$ and integrating in $[1,t]$, for any $t>1$, gives
\begin{equation}
\label{inversefunct2}
\Phi^{-1}\left(\overline{\Phi}(t)\right) \leq \Phi^{-1}\left(\overline{\Phi}(1)\right)t^{\frac{1}{i_\Phi-1}} \quad \forall t>1.
\end{equation}
Thus, by \cite[Proposition 1.4.5 (15)]{G}, \eqref{talentiest}, and \eqref{inversefunct2}, besides reasoning as in \eqref{verif1}, we get
\begin{equation}
\label{rearrangest}
\begin{split}
\|u\|_\infty &= u^*(0) \leq w(0) \leq C\left(\|f\|_{L^\delta(\Omega)}^{\frac{1}{i_\Phi-1}}\int_0^{|\Omega|} t^{\left(\frac{1}{N}-\frac{1}{\delta}\right)\frac{1}{i_\Phi-1}-\frac{1}{N'}} \dt + 1\right).
\end{split}
\end{equation}
The conclusion follows by observing that the integral on the right-hand side of \eqref{rearrangest} is finite, since
\begin{equation}
\label{powercompute2}
\left(\frac{1}{N}-\frac{1}{\delta}\right)\frac{1}{i_\Phi-1} - \frac{1}{N'} > \frac{1-i_\Phi}{N}\frac{1}{i_\Phi-1} - \frac{1}{N'} = -\frac{1}{N}-\frac{1}{N'} = -1.
\end{equation}
\end{proof}

\begin{lemma}
\label{supest}
Let \hyperlink{H1}{${\rm (H_1)}$}, \ref{varstruct}--\ref{degiorgicond}, and \ref{weightscond} be satisfied. Then there exists $M>0$, depending on $L$ of Lemma \ref{energyest} but independent of $\eps\in(0,1)$, such that
\begin{equation*}
\|u_\eps\|_\infty+\|v_\eps\|_\infty \leq M
\end{equation*}
for any $(u_\eps,v_\eps)$ solution to \eqref{regularprob}.
\end{lemma}

\begin{proof}
We prove the claim for $u_\eps$: the same argument can be performed to obtain a uniform bound for $v_\eps$. We split the proof into two parts: the first one is a qualitative information, while the second one is quantitative. Fix any $\eps\in(0,1)$ and set $(u,v):=(u_\eps,v_\eps)$.

\underline{Claim 1}: $u \in L^\infty(\R^N)$. \\
Reasoning as in \cite[Lemma 3.2]{MMM} we have
\begin{equation}
\label{localization}
\int_{\Omega_1} \frac{\phi(|\nabla u|)}{|\nabla u|} \nabla u \nabla \eta \dx \leq C \int_{\Omega_1} \left( \frac{\Upsilon_1(u)}{u} + \frac{\Upsilon_2(v)}{v} + 1 \right) \eta \dx \quad \forall \eta \in \mathcal{D}^{1,\Phi}_0(\R^N).
\end{equation}
Setting $\hat{f}(x):=\frac{\Upsilon_2(v(x))}{v(x)}+1$ and choosing $\eta:=(u-k)_+$, $k>1$, yields
\begin{equation*}
\int_{\Omega_k} \phi(|\nabla u|)|\nabla u| \dx \leq C \int_{\Omega_k} \left( \frac{\Upsilon_1(u)}{u} + \hat{f}(x) \right) (u-k) \dx.
\end{equation*}
This represents the starting point of \cite[Theorem 3.1]{BCM} (see \cite[formulas (4-51)--(4-53)]{BCM}). We also notice that $\hat{f}\in L^{\delta_1}(\Omega_1)$, where $\delta_1>\frac{N}{i_\Phi}$ stems from \ref{degiorgicond}: indeed, by \eqref{powercond2}, \eqref{fundineq2}, Lemma \ref{energyest}, and \eqref{embedding}, we obtain
\begin{equation}
\label{hatfunifbound}
\begin{split}
\int_{\Omega_1} \left(\frac{\Upsilon_2(v)}{v}+1\right)^{\delta_1} \dx &\leq C \left[\int_{\Omega_1} (t^{\delta_1} \circ \overline{\Upsilon}_2^{-1} \circ \Upsilon_2)(v) \dx + 1\right] \\
&\leq C \left[\int_{\Omega_1} \Psi_*(v) \dx + 1\right] < +\infty.
\end{split}
\end{equation}
Then, according to \cite[formulas (3.8)--(3.9)]{BCM} with $M(t):=t^{\delta_1}$, our claim is proved once we show that
\begin{equation*}
\int_0^1 \Phi^{-1} \left( \frac{1}{s} \left[ \int_0^s \overline{\Phi}(r^{\frac{1}{N}-\frac{1}{\delta_1}}) \dr + \int_s^{+\infty} \overline{\Phi}(r^{-\frac{1}{N'}}s^{1-\frac{1}{\delta_1}}) \dr \right] \right) s^{-\frac{1}{N'}} \ds < +\infty.
\end{equation*}
Without loss of generality, we can suppose $\delta_1<N$. Thus, taking into account \eqref{youngind}, we have
\begin{equation}
\label{cianchiest1}
\int_0^s \overline{\Phi}(r^{\frac{1}{N}-\frac{1}{\delta_1}}) \dr \leq C \int_0^s r^{i_\Phi'(\frac{1}{N}-\frac{1}{\delta_1})} \dr
\end{equation}
and, also observing that $r^{-\frac{1}{N'}}s^{1-\frac{1}{\delta_1}}<1$ if and only if $r>s^{N'\left(1-\frac{1}{\delta_1}\right)}$,
\begin{equation}
\label{cianchiest2}
\begin{split}
&\int_s^{+\infty} \overline{\Phi}(r^{-\frac{1}{N'}}s^{1-\frac{1}{\delta_1}}) \dr \\
&= \int_s^{s^{N'\left(1-\frac{1}{\delta_1}\right)}} \overline{\Phi}(r^{-\frac{1}{N'}}s^{1-\frac{1}{\delta_1}}) \dr + \int_{s^{N'\left(1-\frac{1}{\delta_1}\right)}}^{+\infty} \overline{\Phi}(r^{-\frac{1}{N'}}s^{1-\frac{1}{\delta_1}}) \dr \\
&\leq C\left[ s^{i_\Phi'\left(1-\frac{1}{\delta_1}\right)} \int_s^{s^{N'\left(1-\frac{1}{\delta_1}\right)}} r^{-\frac{i_\Phi'}{N'}} \dr + s^{s_\Phi'\left(1-\frac{1}{\delta_1}\right)} \int_{s^{N'\left(1-\frac{1}{\delta_1}\right)}}^{+\infty} r^{-\frac{s_\Phi'}{N'}} \dr \right] \\
&\leq C\left[ s^{i_\Phi'\left(1-\frac{1}{\delta_1}\right)+1-\frac{i_\Phi'}{N'}} + s^{s_\Phi'\left(1-\frac{1}{\delta_1}\right)+N'\left(1-\frac{1}{\delta_1}\right)\left(1-\frac{s_\Phi'}{N'}\right)} \right] \\
&= C\left[ s^{i_\Phi'\left(\frac{1}{N}-\frac{1}{\delta_1}\right)+1} + s^{N'\left(1-\frac{1}{\delta_1}\right)} \right].
\end{split}
\end{equation}
Recalling $\frac{N}{i_\Phi}<\delta_1<N$ and \eqref{powercompute2} yields
\begin{equation}
\label{powercompute3}
0>N'\left(1-\frac{1}{\delta_1}\right)-1>i_\Phi'\left(\frac{1}{N}-\frac{1}{\delta_1}\right)>-\frac{i_\Phi}{N}>-1.
\end{equation}
Thus, putting \eqref{cianchiest1}--\eqref{cianchiest2} together and exploiting \eqref{invgrowth} for $\Lambda=\Phi$, besides using \eqref{powercompute3}, yield
\begin{equation*}
\begin{split}
&\int_0^1 \Phi^{-1} \left( \frac{1}{s} \left[ \int_0^s \overline{\Phi}(r^{\frac{1}{N}-\frac{1}{\delta_1}}) \dr + \int_s^{+\infty} \overline{\Phi}(r^{-\frac{1}{N'}}s^{1-\frac{1}{\delta_1}}) \dr \right] \right) s^{-\frac{1}{N'}} \ds \\
&\leq C \int_0^1 \Phi^{-1} \left(s^{i_\Phi'\left(\frac{1}{N}-\frac{1}{\delta_1}\right)} + s^{N'\left(1-\frac{1}{\delta_1}\right)-1}\right) s^{-\frac{1}{N'}} \ds \\
&\leq C \int_0^1 \Phi^{-1} \left(s^{i_\Phi'\left(\frac{1}{N}-\frac{1}{\delta_1}\right)}\right) s^{-\frac{1}{N'}} \ds \\
&\leq C \int_0^1 s^{(\frac{1}{N}-\frac{1}{\delta_1})\frac{1}{i_\Phi-1}-\frac{1}{N'}} \ds < +\infty,
\end{split}
\end{equation*}
proving the claim.

\underline{Claim 2}: There exists $M>0$, depending on $L$ but not on $\eps$, such that $\|u\|_\infty \leq M$. \\
Let us consider the unique solution $w\in \mathcal{D}^{1,\Phi}_0(\Omega_1)$ to
\begin{equation}
\label{auxprob}
\left\{
\begin{alignedat}{2}
-\Delta_\Phi w &= \tilde{f}(x) \quad &&\mbox{in} \;\; \Omega_1, \\
w &= 0 \quad &&\mbox{on} \;\; \partial \Omega_1,
\end{alignedat}
\right.
\end{equation}
where $\tilde{f}(x) := \frac{\Upsilon_1(u(x))}{u(x)} + \frac{\Upsilon_2(v(x))}{v(x)} + 1$. Existence and uniqueness of $w$ are ensured by Minty-Browder's theorem \cite[Theorem 5.16]{B}, jointly with \cite[Lemma 2.6]{CGL} and $\tilde{f}\in L^{\delta_1}(\Omega_1)$: this summability of $\tilde{f}$ comes from $\hat{f}\in L^{\delta_1}(\Omega_1)$ and $\frac{\Upsilon_1(u(\cdot))}{u(\cdot)}\in L^\infty(\R^N)$, since $u\in L^\infty(\R^N)$ by Claim 1. Choosing $\eta:=(u-w-1)_+$ in \eqref{localization} we have
\begin{equation}
\label{comparison1}
\int_{\Omega_1} \frac{\phi(|\nabla u|)}{|\nabla u|} \nabla u \nabla (u-w-1)_+ \dx \leq C \int_{\Omega_1} \left( \frac{\Upsilon_1(u)}{u} + \frac{\Upsilon_2(v)}{v} + 1 \right) (u-w-1)_+ \dx
\end{equation}
while, testing \eqref{auxprob} with $(u-w-1)_+$, we deduce
\begin{equation}
\label{comparison2}
\int_{\Omega_1} \frac{\phi(|\nabla w|)}{|\nabla w|} \nabla w \nabla (u-w-1)_+ \dx = C \int_{\Omega_1} \left( \frac{\Upsilon_1(u)}{u} + \frac{\Upsilon_2(v)}{v} + 1 \right) (u-w-1)_+ \dx.
\end{equation}
Subtracting \eqref{comparison1}--\eqref{comparison2} term by term we get
\begin{equation*}
\int_{\{u>w+1\}} \left( \frac{\phi(|\nabla u|)}{|\nabla u|} \nabla u - \frac{\phi(|\nabla w|)}{|\nabla w|} \nabla w \right) (\nabla u -\nabla w) \dx \leq 0.
\end{equation*}
The strict monotonicity of the $\Phi$-Laplacian operator (see, e.g., \cite[Lemma 2.6]{CGL}) produces $u \leq w+1$ in $\Omega_1$. In particular,
\begin{equation}
\label{supcomp1}
\|u\|_{L^\infty(\Omega_1)} \leq \|w+1\|_{L^\infty(\Omega_1)} \leq \|w\|_{L^\infty(\Omega_1)}+1.
\end{equation}
Chebichev's inequality entails
\begin{equation*}
|\Omega_1| \leq C \int_{\R^N} \Phi_*(u) \dx,
\end{equation*}
which is uniformly bounded in terms of $L$ by \eqref{embedding} and Lemma \ref{energyest}. We stress the fact that $|\Omega_1|$ does not depend on $u$ but only on $L$, so it is independent on $\eps$. We are in the position to apply Lemma \ref{talenti} to $w$ and deduce
\begin{equation}
\label{supcomp2}
\|w\|_{L^\infty(\Omega_1)} \leq C \left(\|\tilde{f}\|_{L^{\delta_1}(\Omega_1)}^{\frac{1}{i_\Phi-1}} + 1\right).
\end{equation}
Joining \eqref{supcomp1}--\eqref{supcomp2} together and using the definition of $\tilde{f}$ we get
\begin{equation*}
\|u\|_{L^\infty(\Omega_1)} \leq C\left( \left\|\frac{\Upsilon_1(u)}{u}\right\|_{L^{\delta_1}(\Omega_1)}^{\frac{1}{i_\Phi-1}} + \|\hat{f}\|_{L^{\delta_1}(\Omega_1)}^{\frac{1}{i_\Phi-1}} + 1\right),
\end{equation*}
which can be rewritten, using \eqref{factor} and the uniform boundedness of $\hat{f}$ in $L^{\delta_1}(\Omega_1)$ proved in \eqref{hatfunifbound}, as
\begin{equation}
\label{improvement}
\|u\|_{L^\infty(\Omega_1)} \leq C\left( \|u^{s_{\Upsilon_1}-1}\|_{L^{\delta_1}(\Omega_1)}^{\frac{1}{i_\Phi-1}} + 1\right).
\end{equation}
Let us consider $\tau>0$ such that $t^{(s_{\Upsilon_1}-i_\Phi+\tau)\delta_1}<\Phi_*$ (which is possible according to \eqref{powercond1}, taking a smaller $\delta_1$ if necessary). Then \eqref{improvement}, besides observing that $\|u\|_{L^\infty(\Omega_1)}>1$ provided $|\Omega_1|>0$, gives
\begin{equation}
\label{improvement2}
\begin{split}
\|u\|_{L^\infty(\Omega_1)} &\leq C\left( \|u\|_{L^\infty(\Omega_1)}^{\frac{i_\Phi-1-\tau}{i_\Phi-1}} \|u^{s_{\Upsilon_1}-i_\Phi+\tau}\|_{L^{\delta_1}(\Omega_1)}^{\frac{1}{i_\Phi-1}} + 1\right) \\
&\leq C\|u\|_{L^\infty(\Omega_1)}^{1-\frac{\tau}{i_\Phi-1}} \left(\|u\|_{L^{(s_{\Upsilon_1}-i_\Phi+\tau)\delta_1}(\Omega_1)}^{\frac{s_{\Upsilon_1}-i_\Phi+\tau}{i_\Phi-1}} + 1\right) \\
&\leq C\|u\|_{L^\infty(\Omega_1)}^{1-\frac{\tau}{i_\Phi-1}} \left(\|u\|_{L^{\Phi_*}(\Omega_1)}^{\frac{s_{\Upsilon_1}-i_\Phi+\tau}{i_\Phi-1}} + 1\right).
\end{split}
\end{equation}
Accordingly, by the embedding \eqref{embedding} and Lemma \ref{energyest}, estimate \eqref{improvement} can be improved to
\begin{equation*}
\|u\|_{L^\infty(\Omega_1)}^{\frac{\tau}{i_\Phi-1}} \leq C \left(L^{\frac{s_{\Upsilon_1}-i_\Phi+\tau}{i_\Phi-1}} + 1\right).
\end{equation*}
\end{proof}

\begin{lemma}
\label{belowest}
Suppose \hyperlink{H1}{${\rm (H_1)}$}--\hyperlink{H2}{${\rm (H_2)}$}. Then for all $r>0$ there exists $c_r>0$, independent of $\eps\in(0,1)$, such that $\min\{u_\eps,v_\eps\} \geq c_r$ in $B_r$ for all $(u_\eps,v_\eps)$ solution to \eqref{regularprob}.
\end{lemma}

\begin{proof}
Fix any $\eps\in(0,1)$. According to \ref{harnackcond} and Lemma \ref{supest}, $(u,v):=(u_\eps,v_\eps)$ satisfies
\begin{equation}
\label{supersol}
\left\{
\begin{alignedat}{2}
-\Delta_{\Phi} u &\geq C^{-1}h(x)v^{\nu_2} &&\quad \mbox{in}\;\; \R^N, \\
-\Delta_{\Psi} v &\geq C^{-1}k(x)u^{\nu_1} &&\quad \mbox{in}\;\; \R^N,
\end{alignedat}
\right.
\end{equation}
for a suitable $C>0$ depending on the parameter $M$ stemming from Lemma \ref{supest}.

Fix $\eta \in C^\infty_c(\R^N)$ such that $\eta \equiv 1$ in $B_r$ and $\eta\equiv 0$ outside $B_R$, with $0<r<R<+\infty$. Set $w_l:=u+l$, for any $l\in(0,1)$. Testing the first equation of \eqref{supersol} with $w_l^{-\sigma}\eta^a$, being $\sigma\in(0,\frac{1}{s_{\overline{\Phi}}-1})$ and $a>i_{\overline{\Phi}}'$, besides noticing that Lemma \ref{supest} and \ref{weightscond} yield respectively $w_l \leq M+1$ and $h\geq c_R$ in $B_R$, we get
\begin{equation}
\label{belowtest1}
\begin{split}
&a \int_{B_R} w_l^{-\sigma}\eta^{a-1} \phi(|\nabla u|) \frac{\nabla u \nabla \eta}{|\nabla u|} \dx -\sigma \int_{B_R} w_l^{-\sigma-1}\eta^a \phi(|\nabla u|)|\nabla u| \dx \\
&\geq C^{-1} \int_{B_R} h w_l^{-\sigma}v^{\nu_2}\eta^a \dx \geq C^{-1} \int_{B_r} v^{\nu_2} \dx.
\end{split}
\end{equation}
Pick any $\mu\in(0,1)$. The fist integral of \eqref{belowtest1} can be estimated via \eqref{indices}, Young's inequality, the convexity of $\overline{\Phi}$, \eqref{factor}, \eqref{youngind}, and \eqref{fundineq2} as
\begin{equation}
\label{reabsorb}
\begin{split}
&\int_{B_R} w_l^{-\sigma}\eta^{a-1} \phi(|\nabla u|) |\nabla \eta| \dx \leq s_\Phi \int_{B_R} w_l^{-\sigma}\eta^{a-1} \frac{\Phi(|\nabla u|)}{|\nabla u|} |\nabla \eta| \dx \\
&= s_\Phi \int_{B_R} \left(\mu \eta^{a i_{\overline{\Phi}}^{-1}} w_l^{-\frac{\sigma+1}{s_{\overline{\Phi}}}} \frac{\Phi(|\nabla u|)}{|\nabla u|}\right) \left(\mu^{-1} \eta^{a-1-a i_{\overline{\Phi}}^{-1}}|\nabla \eta| w_l^{\frac{\sigma+1}{s_{\overline{\Phi}}}-\sigma}\right) \dx \\
&\leq s_\Phi \left[ \int_{B_R} \overline{\Phi}\left(\mu \eta^{a i_{\overline{\Phi}}^{-1}} w_l^{-\frac{\sigma+1}{s_{\overline{\Phi}}}} \frac{\Phi(|\nabla u|)}{|\nabla u|}\right) \dx + \int_{B_R} \Phi\left(\mu^{-1} \eta^{a-1-a i_{\overline{\Phi}}^{-1}}|\nabla \eta| w_l^{\frac{\sigma+1}{s_{\overline{\Phi}}}-\sigma}\right) \dx \right] \\
&\leq C\mu \int_{B_R} \eta^a w_l^{-\sigma-1} \overline{\Phi}\left(\frac{\Phi(|\nabla u|)}{|\nabla u|}\right) \dx + C_{\mu,\eta} \int_{B_R} w_l^{i_\Phi\left(\frac{\sigma+1}{s_{\overline{\Phi}}}-\sigma\right)} \dx \\
&\leq C\mu \int_{B_R} \eta^a w_l^{-\sigma-1} \phi(|\nabla u|)|\nabla u| \dx + C_{\mu,\eta} \int_{B_R} w_l^{i_\Phi\left(\frac{\sigma+1}{s_{\overline{\Phi}}}-\sigma\right)} \dx,
\end{split}
\end{equation}
for a suitable $C>0$ depending on $\Phi$ and $M$. By \eqref{belowtest1}--\eqref{reabsorb} we get
\begin{equation*}
\begin{aligned}
C^{-1} \int_{B_r} v^{\nu_2} \dx &\leq \left(Ca\mu-\sigma\right)\int_{B_R} w_l^{-\sigma-1}\eta^a \phi(|\nabla u|)|\nabla u| \dx + C_{\mu,\eta} \int_{B_R} w_l^{i_\Phi\left(\frac{\sigma+1}{s_{\overline{\Phi}}}-\sigma\right)} \dx.
\end{aligned}
\end{equation*}
Choosing $\mu<\frac{\sigma}{Ca}$ yields
\begin{equation*}
\int_{B_r} v^{\nu_2} \dx \leq C \int_{B_R} w_l^{i_\Phi\left(\frac{\sigma+1}{s_{\overline{\Phi}}}-\sigma\right)} \dx,
\end{equation*}
where $C$ depends on $M,\mu,\eta,\Phi,a,\sigma$. Reasoning as in \eqref{belowtest1}--\eqref{reabsorb}, besides letting $l\to 0^+$, we obtain
\begin{equation}
\label{intest3}
\int_{B_r} v^{\nu_2} \dx \leq C \int_{B_R} u^{i_\Phi\left(\frac{\sigma+1}{s_{\overline{\Phi}}}-\sigma\right)} \dx
\end{equation}
and
\begin{equation}
\label{intest4}
\int_{B_r} u^{\nu_1} \dx \leq C \int_{B_R} v^{i_\Psi\left(\frac{\sigma+1}{s_{\overline{\Psi}}}-\sigma\right)} \dx.
\end{equation}
Let us set
\begin{equation}
\label{hatsigmadef}
\tau_1:=\min\left\{\frac{i_\Phi}{s_{\overline{\Phi}}},\frac{N(i_\Phi-1)}{N-i_\Phi}\right\}-\hat{\sigma} \quad \mbox{and} \quad \tau_2:=\min\left\{\frac{i_\Psi}{s_{\overline{\Psi}}},\frac{N(i_\Psi-1)}{N-i_\Psi}\right\}-\hat{\sigma},
\end{equation}
where $\hat{\sigma}>0$ is chosen such that $\tau_i>0$, $i=1,2$, and
\begin{equation}
\label{sigmadef}
\nu_1\nu_2<\tau_1\tau_2,
\end{equation}
which is possible according to \ref{harnackcond}. Then fix $\sigma$ as above fulfilling
\begin{equation*}
i_\Phi\left(\frac{\sigma+1}{s_{\overline{\Phi}}}-\sigma\right)>\frac{i_\Phi}{s_{\overline{\Phi}}}-\hat{\sigma} \quad \mbox{and} \quad i_\Psi\left(\frac{\sigma+1}{s_{\overline{\Psi}}}-\sigma\right)>\frac{i_\Psi}{s_{\overline{\Psi}}}-\hat{\sigma}.
\end{equation*}
Exploiting Lemma \ref{supest} again and \eqref{hatsigmadef}--\eqref{sigmadef}, estimates \eqref{intest3}--\eqref{intest4} yield
\begin{equation}
\label{intest5}
\int_{B_r} v^{\nu_2} \dx \leq C M^{i_\Phi\left(\frac{\sigma+1}{s_{\overline{\Phi}}}-\sigma\right)-\tau_1} \int_{B_R} u^{\tau_1} \dx \leq C \int_{B_R} u^{\tau_1} \dx
\end{equation}
and
\begin{equation}
\label{intest6}
\int_{B_r} u^{\nu_1} \dx \leq C M^{i_\Psi\left(\frac{\sigma+1}{s_{\overline{\Psi}}}-\sigma\right)-\tau_2} \int_{B_R} v^{\tau_2} \dx \leq C \int_{B_R} v^{\tau_2} \dx.
\end{equation}
Using the weak Harnack inequality \cite[Theorem 1.4 case 1]{BHHK} (with $s=\infty$ and $l_0=\tau_1$) and \eqref{intest5}, besides recalling \eqref{hatsigmadef} (which ensures that $\tau_1$ is strictly less than $l(i_\Phi)$ defined in \cite[p.792]{BHHK}) and \cite[Remark 1.6]{BHHK}, we get
\begin{equation}
\label{belowest1}
\begin{split}
\inf_{B_r} u &\geq cR^{-\frac{N}{\tau_1}} \left(\int_{B_R} u^{\tau_1} \dx\right)^{\frac{1}{\tau_1}} \geq cR^{-\frac{N}{\tau_1}} \left(\int_{B_r} v^{\nu_2} \dx\right)^{\frac{1}{\tau_1}} \geq c\left(\frac{r}{R}\right)^{\frac{N}{\tau_1}} \left(\inf_{B_r} v\right)^{\frac{\nu_2}{\tau_1}}.
\end{split}
\end{equation}
Repeating the same argument, from \eqref{intest6} we obtain
\begin{equation}
\label{belowest2}
\inf_{B_r} v \geq c\left(\frac{r}{R}\right)^{\frac{N}{\tau_2}} \left(\inf_{B_r} u\right)^{\frac{\nu_1}{\tau_2}}.
\end{equation}
Let $R=2r$. Concatenating \eqref{belowest1}--\eqref{belowest2} and bearing in mind that $c$ depends on $\eta$, and a fortiori on $r$, yield
\begin{equation*}
\inf_{B_r} u \geq c_r \left(\inf_{B_r} u\right)^{\frac{\nu_1\nu_2}{\tau_1\tau_2}}. 
\end{equation*}
The conclusion then follows by \eqref{sigmadef}.
\end{proof}
\begin{rmk}
\label{nonshift}
An inspection of the proofs of Lemmas \ref{energyest}, \ref{supest}, and \ref{belowest} reveals that the estimates proved there hold true also for $\eps=0$, that is, for the solutions to problem \eqref{prob}.
\end{rmk}

\subsection{Conclusion}

\begin{lemma}
\label{distrsol}
Under \hyperlink{H1}{${\rm (H_1)}$}--\hyperlink{H2}{${\rm (H_2)}$}, there exists a distributional solution $(u,v)\in C^{1,\tau}_{\rm loc}(\R^N)$ to \eqref{prob}, for a suitable $\tau\in(0,1]$.
\end{lemma}

\begin{proof}
Let us set, for all $n\in\N$, $(u_n,v_n):=(u_{\eps_n},v_{\eps_n})$, where $\eps_n \to 0^+$ and $(u_{\eps_n},v_{\eps_n})$ comes from Theorem \ref{regexistence}. We also set $\hat{f}_n(x):=h(x)f(x,u_n+\eps_n,v_n+\eps_n)$ and $\hat{g}_n(x):=k(x)g(x,u_n+\eps_n,v_n+\eps_n)$. By Lemma \ref{energyest} we infer that $\{(u_n,v_n)\}$ is bounded in $X$, while Lemmas \ref{supest} and \ref{belowest} ensure that $\{\hat{f}_n\}$ and $\{\hat{g}_n\}$ are bounded in $L^\infty_{\rm loc}(\R^N)$: indeed, for any $r>0$,
\begin{equation*}
\hat{f}_n(x) \leq h(x)[(c_r^{-\alpha}+1)((M+1)^{s_{\Upsilon_2}-1}+1)+(M+1)^{s_{\Upsilon_1}-1}] \in L^\infty(B_r),
\end{equation*}
and the same holds for $\hat{g}_n$. Hence Lieberman's regularity theory guarantees that $\{(u_n,v_n)\}$ is bounded in $C^{1,\tau_0}_{\rm loc}(\R^N)^2$ for a suitable $\tau_0\in(0,1]$. Then, fixed any $\tau\in(0,\tau_0)$, Ascoli-Arzelà's theorem furnishes $(u,v)\in C^{1,\tau}_{\rm loc}(\R^N)^2$ such that $(u_n,v_n) \to (u,v)$ in $C^{1,\tau}_{\rm loc}(\R^N)^2$. Passing to the limit in the distributional formulation of \eqref{regularprob}, via uniform convergence and dominated convergence on the left-hand and right-hand sides respectively, reveals that $(u,v)$ is a distributional solution to \eqref{prob}.
\end{proof}

\begin{thm}
\label{weaksol}
Suppose \hyperlink{H1}{${\rm (H_1)}$}--\hyperlink{H2}{${\rm (H_2)}$}. Then there exists a weak solution $(u,v)\in C^{1,\tau}_{\rm loc}(\R^N)$ to \eqref{prob}, for a suitable $\tau\in(0,1]$.
\end{thm}

\begin{proof}
This proof is patterned after \cite[Lemma 4.2]{GMM}, so here we only sketch it. Lemma \ref{distrsol} furnishes a distributional solution $(u,v)\in C^{1,\tau}_{\rm loc}(\R^N)$ to \eqref{prob}. Now we show that it is actually a weak solution to the problem. To this end, we reason for the first equation and pick any test function $\eta\in \mathcal{D}^{1,\Phi}_0(\R^N)$ and a cut-off function $\theta\in C^\infty(\R_+)$ such that
\begin{equation*}
\theta(t) := \left\{
\begin{alignedat}{1}
1, \quad &\mbox{in} \;\; [0,1], \\
\mbox{decreasing}, \quad &\mbox{in} \;\; ]1,2[, \\
0, \quad &\mbox{in} \;\; [2,+\infty).
\end{alignedat}
\right.
\end{equation*}
We split the test function as $\eta = \eta_+-\eta_-$ and set, for every $n,m\in\N$,
\begin{equation*}
\begin{split}
\theta_n := \theta\left(\frac{|\cdot|}{n}\right) \in C^\infty_c(\R^N), \quad \eta_n := \theta_n \eta_+ \in \mathcal{D}^{1,\Phi}_0(\R^N), \quad \eta_{m,n} := \rho_m * \eta_n \in C^\infty_c(\R^N).
\end{split}
\end{equation*}
We notice that $\eta_n \nearrow \eta_+$ and $\eta_{m,n} \to \eta_n$ in $\mathcal{D}^{1,\Phi}_0(\R^N)$ as $m\to\infty$ (see, e.g., \cite[Theorem 3.18.1.1]{KJF}), which readily entails
\begin{equation*}
\lim_{m\to\infty} \int_{\R^N} \phi(|\nabla u|)\frac{\nabla u \nabla \eta_{m,n}}{|\nabla u|} \dx = \int_{\R^N} \phi(|\nabla u|)\frac{\nabla u \nabla \eta_n}{|\nabla u|} \dx.
\end{equation*}
Reasoning as in \cite[Lemma 4.2]{GMM}, besides exploiting Lemmas \ref{supest} and \ref{belowest}, ensures that
\begin{equation*}
\lim_{m\to\infty} \int_{\R^N} hf(u,v)\eta_{m,n} \dx = \int_{\R^N} hf(u,v)\eta_n \dx.
\end{equation*}
Accordingly, we can pass to the limit in
\begin{equation*}
\int_{\R^N} \phi(|\nabla u|)\frac{\nabla u \nabla \eta_{m,n}}{|\nabla u|} \dx = \int_{\R^N} hf(u,v)\eta_{m,n} \dx,
\end{equation*}
that holds true for any $m,n\in\N$, since $(u,v)$ is a distributional solution to \eqref{prob}; we obtain
\begin{equation*}
\int_{\R^N} \phi(|\nabla u|)\frac{\nabla u \nabla \eta_n}{|\nabla u|} \dx = \int_{\R^N} hf(u,v)\eta_n \dx
\end{equation*}
for all $n\in\N$. Moreover, Beppo Levi's monotone convergence theorem ensures
\begin{equation*}
\lim_{n\to\infty} \int_{\R^N} hf(u,v)\eta_n \dx = \int_{\R^N} hf(u,v)\eta_+ \dx.
\end{equation*}

In order to prove
\begin{equation*}
\lim_{n\to\infty} \int_{\R^N} \phi(|\nabla u|)\frac{\nabla u \nabla \eta_n}{|\nabla u|} \dx = \int_{\R^N} \phi(|\nabla u|)\frac{\nabla u \nabla \eta_+}{|\nabla u|} \dx,
\end{equation*}
it suffices to show that $\eta_n \to \eta_+$ in $\mathcal{D}^{1,\Phi}_0(\R^N)$, that is, $\int_{\R^N} \Phi(|\nabla \eta_n - \nabla \eta_+|) \dx$ goes to zero as $n\to\infty$. Noticing that $\supp \theta_n \subseteq \{x\in\R^N: \, |x|<2n\} =: K_n$ and setting $A_n:=K_{2n}\setminus K_n$, we get
\begin{equation}
\label{scaling2}
\begin{split}
\int_{\R^N} \Phi(|\nabla \eta_n - \nabla \eta_+|) \dx &= \int_{\R^N} \Phi(|\eta_+ \nabla \theta_n + \theta_n \nabla \eta_+ - \nabla \eta_+|) \dx \\
&\leq C \left( \int_{\R^N} \Phi((1-\theta_n)|\nabla \eta_+|) \dx + \int_{A_n} \Phi(\eta_+ |\nabla \theta_n|) \dx \right).
\end{split}
\end{equation}
We treat the integrals on the right-hand side of \eqref{scaling2} separately. For the first integral, by Lebesgue's theorem we infer
\begin{equation*}
\lim_{n\to\infty} \int_{\R^N} \Phi((1-\theta_n)|\nabla \eta_+|) \dx = 0,
\end{equation*}
since $\eta_+ \in \mathcal{D}^{1,\Phi}_0(\R^N)$ and $\theta_n \to 1$ in $\R^N$. Regarding the second integral, we show that $\|\eta_+|\nabla \theta_n|\|_{L^\Phi(A_n)} \to 0$ as $n\to\infty$. By Proposition \ref{holderineq} we have
\begin{equation*}
\|\eta_+|\nabla \theta_n|\|_{L^\Phi(A_n)} \leq \|\eta_+\|_{L^{\Phi_*}(A_n)} \|\nabla \theta_n\|_{L^N(\R^N)}.
\end{equation*}
We notice that $\|\eta_+\|_{L^{\Phi_*}(A_n)} \to 0$ as $n\to\infty$, since $\eta_+\in L^{\Phi_*}(\R^N)$ by \eqref{embedding}. Moreover, $\|\nabla \theta_n\|_{L^N(\R^N)}$ is bounded uniformly in $n$: indeed, through a change of variable, we get
\begin{equation*}
\|\nabla \theta_n\|_{L^N(\R^N)}^N = \int_{\R^N} |\nabla \theta_n|^N \dx = \frac{1}{n^N} \int_{\R^N} \left|\theta'\left(\frac{|x|}{n}\right)\right|^N \dx = \int_{\R^N} |\theta'(|x|)|^N \dx <+\infty.
\end{equation*}
\end{proof}

\section{Uniqueness result}
\begin{lemma}
\label{decay}
Suppose \hyperlink{H1}{${\rm (H_1)}$}--\hyperlink{H2}{${\rm (H_2)}$} and \ref{decaycond} to be satisfied. Let $(u,v)$ be a weak solution to \eqref{prob}. Then
\begin{equation*}
\begin{split}
&C^{-1} \int_{|x|}^{+\infty} \phi^{-1}(r^{1-N})\dr\le u(x)\le C\int_{|x|}^{+\infty} \phi^{-1}(r^{1-N})\dr,\\
&C^{-1} \int_{|x|}^{+\infty} \psi^{-1}(r^{1-N})\dr\le v(x)\le C\int_{|x|}^{+\infty} \psi^{-1}(r^{1-N})\dr,
\end{split}
\end{equation*}
for a suitable $C>0$.
\end{lemma}
\begin{proof}
Lemma \ref{belowest} (see also Remark \ref{nonshift}) ensures that 
\begin{equation}\label{localbound}
\inf_{B_1}u= c
\end{equation}
for a suitable $c>0$. Lemma \ref{radialsub} furnishes $\underline{w}\in \mathcal{D}^{1,\Phi}_0(B_1^e)$ such that  
\begin{equation*}
\left\{
\begin{alignedat}{2}
-\Delta_{\Phi} \underline{w} &= 0 &&\quad \mbox{in}\;\; B_1^{e}, \\
\underline{w} &= c &&\quad \mbox{on}\;\; \partial B_1, \\
\underline{w}(x) &\to 0 &&\quad \mbox{as}\;\; |x| \to +\infty.
\end{alignedat}
\right.
\end{equation*}
The weak comparison principle (see, e.g., \cite[Theorem 3.4.1]{PS}) and \eqref{localbound} ensure $\underline{w}\le u$ in $B_1^{e}$. In particular, by \eqref{radialsubrepr} we have
\begin{equation*}
C^{-1} \int_{|x|}^{+\infty} \phi^{-1}(r^{1-N}) \dr \le \underline{w}(x)\le u(x) \quad \forall x\in B_1^e.
\end{equation*}
The same argument produces a lower bound for $v$ in $B_1^e$.

According to Lemma \ref{supest} (besides Remark \ref{nonshift}), we have $\|u\|_\infty+\|v\|_\infty\leq M$ for some $M>0$. Set $l_1:=\theta_1-\alpha\frac{N-i_\Phi}{i_\Phi-1}>N$. Then, using \ref{decaycond}, \ref{growthcond}, and Remark \ref{decayest} as in Lemma \ref{distrsol}, we obtain
\begin{equation*} 
\begin{split}
h(x)f(u(x),v(x)) &\leq C|x|^{-\theta_1}\left[(u(x)^{-\alpha}+1)(v(x)^{s_{\Upsilon_2}-1}+1)+u(x)^{s_{\Upsilon_1}-1}+v(x)^{s_{\Upsilon_2}-1}\right] \\ 
&\leq C|x|^{-l_1-\alpha\frac{N-i_\Phi}{i_\Phi-1}}\left(|x|^{\alpha\frac{N-i_\Phi}{i_\Phi-1}}+1 \right) \leq C|x|^{-l_1},
\end{split}
\end{equation*}
for all $x\in B_1^e$, being $C=C(M,c)$. Lemma \ref{radialsuper} produces a solution $\overline{w}$ to
\begin{equation*}
\left\{
\begin{alignedat}{2}
-\Delta_{\Phi} \overline{w} &= C|x|^{-l_1} &&\quad \mbox{in}\;\; B_1^{e}, \\
\overline{w} &= \hat{M} &&\quad \mbox{on}\;\; \partial B_1, \\
\overline{w}(x) &\to 0 &&\quad \mbox{as}\;\; |x| \to +\infty,
\end{alignedat}
\right.
\end{equation*}
choosing $\hat{M}>M$ sufficiently large, as prescribed in \eqref{boundarycond}. The weak comparison principle ensures that
\begin{equation*}
u(x)\le \overline{w}(x) \le C\int_{|x|}^{+\infty} \phi^{-1}(r^{1-N}) \dr \quad \forall x\in B_1^e.
\end{equation*}
Arguing in the same way, we deduce also an upper bound for $v$ in $B_1^e$.
\end{proof}

\begin{lemma}
\label{convexityofJ}
Let $J:\mathcal{D}^{1,\Theta}_{0}(\R^N)\to [0,+\infty)$ be defined as
\begin{equation*}
J(w)=\int_{\R^N}\Theta(|\nabla w|) \dx, 
\end{equation*} 
where $\Theta$ is a Young function of class $C^2$ and $\theta=\Theta'$. For any fixed $k\in (1,i_\theta+1]$, we consider $\tilde{J}:L^1(\R^N)\to [0,+\infty]$ given by
\begin{equation*}
\tilde{J}(w):=\begin{cases}
J\Big(w^{\frac{1}{k}}\Big), & \text{if $w\ge 0$ \;and\; $w^{\frac{1}{k}}\in \mathcal{D}^{1,\Theta}_{0}(\R^N)$}, \\
+\infty, & \text{otherwise.}
\end{cases}
\end{equation*}
Thus $\dom \tilde{J}\not = \emptyset$ and $\tilde{J}$ is convex. 
\end{lemma}
\begin{proof}
Reasoning as in \cite[Theorem 1.2]{GCS}, we have that $\tilde{J}\not \equiv \infty$. Now we show that $\tilde{J}$ is a convex functional. Take any $w_i\in \dom \tilde{J}$, $i=1,2$, and $\lambda\in[0,1]$. Set $z_i=w_i^{\frac{1}{k}}$, $i=1,2$, and $z_3= (\lambda w_1+(1-\lambda)w_2)^{\frac{1}{k}}$. By \cite[Theorem 1.2]{GCS} we have  
\begin{equation}
\label{convexineq}
|\nabla z_3| \le (\lambda |\nabla z_1|^k+(1-\lambda)|\nabla z_2|^k)^{\frac{1}{k}}.
\end{equation}	
Proposition \ref{convex} and (2.2) guarantee that $\Theta \circ t^{\frac{1}{k}}$ is convex. Hence, by \eqref{convexineq},
\begin{equation*}
\begin{split}
&\tilde J(\lambda w_1+(1-\lambda)w_2)=\int_{\R^N} \Theta 	(|\nabla z_3|) \dx\le \int_{\R^N} \Theta 	((\lambda |\nabla z_1|^k+(1-\lambda)|\nabla z_2|^k)^{\frac{1}{k}}) \dx\\
&\le  \lambda\int_{\R^N}  \Theta ( |\nabla z_1|)\dx+ (1-\lambda)\int_{\R^N}  \Theta(|\nabla z_2|) \dx=\lambda\tilde J( w_1)+(1-\lambda)\tilde J (w_2).
\end{split}
\end{equation*}
\end{proof}

\begin{thm}
\label{uniqueness}
Assume that \hyperlink{H1}{${\rm (H_1)}$}--\hyperlink{H3}{${\rm (H_3)}$} hold. Then \eqref{prob} has a unique weak solution $(u,v)\in X$.
\end{thm}
\begin{proof}
Let $J:X\to \mathbb{Re}^N$ be defined as
\begin{equation*}
J(u,v):=\int_{\R^N}\Phi(|\nabla u|) \dx+\int_{\R^N}\Psi(|\nabla v|) \dx.
\end{equation*} 
Consider $\tilde{J}:L^1(\R^N)\times L^1(\R^N)\to [0,+\infty]$ given by
\begin{equation*}
\tilde{J}(u,v):=\begin{cases}
J\Big(u^{\frac{1}{i_\phi+1}},v^{\frac{1}{i_\psi+1}}\Big), & \text{if \;$u,v\ge 0$ \;and\; $\Big(u^{\frac{1}{i_\phi+1}}, v^{\frac{1}{i_\psi+1}}\Big)\in X$}, \\
+\infty, & \text{otherwise.}
\end{cases}
\end{equation*}
Applying Lemma \ref{convexityofJ} componentwise, we infer  $\dom \tilde{J}\not = \emptyset$ and that $\tilde{J}$ is a convex functional. In order to prove uniqueness, we suppose that there exist $(u_i,v_i)\in \dom \tilde{J} $, $i=1,2$, weak solutions to \eqref{prob}. From Lemma \ref{decay} we get
\begin{equation*}
\begin{split}
(\xi_1, \eta_1)&:=\left(\frac{u_1^{i_\phi+1}-u_2^{i_\phi+1}}{u_1^{i_\phi}},\frac{v_1^{i_\psi+1}-v_2^{i_\psi+1}}{v_1^{i_\psi}}\right)\in X, \\ (\xi_2, \eta_2)&:=\left(\frac{u_1^{i_\phi+1}-u_2^{i_\phi+1}}{u_2^{i_\phi}},\frac{v_1^{i_\psi+1}-v_2^{i_{\psi+1}}}{v_2^{i_\psi}}\right)\in X.
\end{split}
\end{equation*}
The convexity of $\tilde{J}$ yields
\begin{equation}
\label{convexityoperator}
\begin{split}
0 &\le \langle\nabla \tilde J (u_1^{i_\phi+1},v_1^{i_\psi+1}) -\nabla \tilde J(u_2^{i_\phi+1},v_2^{i_\psi+1}), (u_1^{i_\phi+1}-u_2^{i_\phi+1},v_1^{i_\psi+1}-v_2^{i_\psi+1}) \rangle\\ 
&=\int_{\R^N}[\phi(|\nabla u_1|)\nabla u_1\nabla \xi_1 -\phi(|\nabla u_2|)\nabla u_2\nabla \xi_2] \dx\\
&\quad + \int_{\R^N}[\psi(|\nabla v_1|)\nabla v_1\nabla\eta_1 -\psi(|\nabla v_2|)\nabla v_2\nabla\eta_2]\dx.
\end{split}
\end{equation} 
Since $(u_i,v_i)$ are weak solutions to \eqref{prob}, \eqref{convexityoperator} and \ref{diazsaacond} imply

\begin{equation*}
\begin{split}
0&\le \int_{\R^N}[\phi(|\nabla u_1|)\nabla u_1\nabla \xi_1 -\phi(|\nabla u_2|)\nabla u_2\nabla \xi_2] \dx \\
&\quad + \int_{\R^N}[\psi(|\nabla v_1|)\nabla v_1\nabla\eta_1 -\psi(|\nabla v_2|)\nabla v_2\nabla\eta_2]\dx \\
&=\int_{\R^N}h[f(u_1, v_1)\xi_1-f(u_2, v_2)\xi_2]\dx + \int_{\R^N}k[g(u_1, v_1)\eta_1-g(u_2, v_2)\eta_2]\dx\\
&=\int_{\R^N}h\bigg[\frac{f(u_1, v_1)}{u_1^{i_\phi}}-\frac{f(u_2, v_2)}{u_2^{i_\phi}}\bigg](u_1^{i_\phi+1}-u_2^{i_\phi+1})\dx\\
&\quad +\int_{\R^N}k\bigg[\frac{g(u_1, v_1)}{v_1^{i_\psi}}-\frac{g(u_2, v_2)}{v_2^{i_\psi}}\bigg](v_1^{i_\psi+1}-v_2^{i_\psi+1})\dx\le 0.
\end{split}
\end{equation*}
Hence $u_1\equiv u_2$ and $v_1\equiv v_2$ in $\R^N$, concluding the proof.
\end{proof}

\section*{Acknowledgments}

We warmly thank Prof. Andrea Cianchi, for pointing out the proof of Proposition \ref{holderineq}, and Prof. Sunra Mosconi, for some insights on Lemmas \ref{talenti}--\ref{supest}.

The authors are members of \textit{Gruppo Nazionale per l'Analisi Matematica, la Probabilità e le loro Applicazioni} (GNAMPA) of the \textit{Istituto Nazionale di Alta Matematica} (INdAM).\\
U.Guarnotta was supported by the following research projects: 1) PRIN 2017 `Nonlinear Differential Problems via Variational, Topological and Set-valued Methods' (Grant no. 2017AYM8XW) of MIUR; 2) `MO.S.A.I.C.' PRA 2020--2022 `PIACERI' Linea 3 of the University of Catania; 3) GNAMPA-INdAM Project CUP\_E55F22000270001.

\vspace{0.3cm}

\noindent
\textbf{Conflict of interest statement.} On behalf of all authors, the corresponding author states that there is no conflict of interest. \\
\textbf{Data availability statement.} Data sharing not applicable to this article as no datasets were generated or analysed during the current study.

\begin{small}

\end{small}

\end{document}